\documentclass[12pt]{amsart}
\usepackage{amsmath,amssymb,amscd,array, mathrsfs }

\usepackage{amsmath,amscd,amssymb,amsthm,array}
\usepackage{color}

\usepackage{amsmath,amssymb,mathrsfs,amsthm, tikz-cd,mathrsfs}
\usepackage{comment}
\def\url#1{\expandafter\s

\tring\csname #1\endcsname}

\def\mmat #1,#2,#3,#4,{\text{\small\arraycolsep=3pt $
\begin{pmatrix}#1&#2\\#3&#4\end{pmatrix}$}}

\usepackage{hyperref}
\usepackage{capt-of}

\usepackage{multirow}
\usepackage[all]{xy}

\usepackage{comments}
\newComments\SBe{Said}{blue}
\newComments\SBo{Sofiane}{blue}
\newComments\AM{Nacer}{blue}
\newComments\DL{DL}{red}
\newComments\QEh{QEh}{blue}

\def\mmat #1,#2,#3,#4,{\text{\small\arraycolsep=3pt $
\begin{pmatrix}#1&#2\\#3&#4\end{pmatrix}$}}

\usepackage{lscape}
\usepackage{tikz-cd}
\usepackage{enumerate}

\usepackage{DLdef1}
\usepackage{multicol}

\def\mmat #1,#2,#3,#4,{\text{\small\arraycolsep=3pt $
\begin{pmatrix}#1&#2\\#3&#4\end{pmatrix}$}}

\renewcommand {\ssbegin}[2][*]
 {\refstepcounter{subsection}%
\if#1*
\addcontentsline{toc}{subsection}{\thesubsection.\hskip 1pc #2}%
\else
\addcontentsline{toc}{subsection}{\thesubsection.\hskip 1pc #2. #1}%
\fi
 \def \secno {\gdef \secno {}{\ssecfont
\thesubsection.\hskip 2ex}%
 }%
 \begin{#2}}

\renewcommand {\sssbegin}[2][*]
 {\refstepcounter{subsubsection}
\if#1*
\addcontentsline{toc}{subsubsection}{\thesubsubsection.\hskip 1pc #2}%
\else
\addcontentsline{toc}{subsubsection}{\thesubsubsection.\hskip 1pc #2. #1}
\fi
 \def \secno {\gdef \secno {}{\ssecfont \thesubsubsection.\hskip 2ex}%
 }%
 \begin{#2}}

\renewcommand {\parbegin}[2][*]
 {\refstepcounter{paragraph}
\if#1*
\addcontentsline{toc}{paragraph}{\theparagraph.\hskip 1pc #2}%
\else
\addcontentsline{toc}{paragraph}{\theparagraph.\hskip 1pc #2. #1}
\fi
 \def \secno {\gdef \secno {}{\ssecfont \theparagraph.\hskip 2ex}%
 }%
 \begin{#2}}

\rmnameii{etr}{etr}
\rmnameii{evv}{ev}

\DeclareMathOperator{\K}{\mathbb{K}}
\DeclareMathOperator{\R}{\mathbb{R}}

\DeclareMathOperator{\spa}{span}

\newcommand {\A}{{\cal{A}}}

\newcommand{\h}{\mathfrak{h}}
\newcommand{\Ll}{{\mathrm{L}}}
\newcommand{\Rr}{{\mathrm{R}}}

    \newcommand{\Om}{\Omega}

\newcommand{\al}{\alpha}
\newcommand{\be}{\beta}

\newcommand{\la}{\lambda}

\newcommand{\om}{\omega}
\newcommand{\prs}{\langle \cdot,\cdot\rangle}
\newcommand{\g}{\mathfrak{g}}
\def\br{[\cdot,\cdot]}
\newcommand{\esp}{\quad\mbox{and}\quad}

\begin{document}

\title[Pre-symplectic left-symmetric algebras]{Pre-symplectic left-symmetric algebras  }

\author{Sa\"id Benayadi}
\address {Université de Lorraine, Laboratoire IECL, CNRS-UMR 7502, UFR MIM, 3 rue
Augustin Fresnel, BP 45112, 57073 Metz Cedex 03, France}
\email{said.benayadi@univ-lorraine.fr}

\author{Hamza El Ouali}
\address{Division of Science and Mathematics, New York University Abu Dhabi, P.O. Box 129188, Abu Dhabi, United Arab Emirates.}
\email{hamza.el.ouali@nyu.edu}

\thanks{H.E. was supported by the grant NYUAD-065}




\begin{abstract}
A pre-symplectic left-symmetric algebra $\mathcal{A}$ is a left-symmetric algebra endowed with a nondegenerate skew-symmetric bilinear form $\omega$ such that all left multiplication operators are symmetric with respect to $\omega$. In this setting, the underlying subadjacent Lie algebra $(\mathcal{A}^{-},\omega)$ forms a flat $T$-symplectic Lie algebra. This paper provides a systematic investigation into the structural properties of pre-symplectic left-symmetric algebras. In particular, we introduce a distinguished subclass termed \emph{Milnor pre-symplectic algebras}, and prove that any pre-symplectic left-symmetric algebra whose commutator ideal is nondegenerate necessarily belongs to this subclass.  Next, we investigate the Levi-Civita product associated with symplectic Lie algebras. We show that this product always yields a right-symmetric algebra, and we prove that it forms a left-symmetric algebra if and only if it is associative. Furthermore, we provide a characterization of symplectic Lie algebras in terms of representations of left-symmetric algebras, and conclude by establishing a construction method for these structures known as the $T^*$-extension.

Furthermore, we develop a double extension procedure for pre-symplectic left-symmetric algebras by means of commutative associative algebras. We show that every such algebra with a degenerate commutator ideal can be reconstructed via this extension process. More generally, we show that any pre-symplectic left-symmetric algebra is either a Milnor pre-symplectic algebra or can be obtained through a finite sequence of successive double extensions starting from a Milnor pre-symplectic algebra. As a concrete application of these structural results, we provide a complete classification of pre-symplectic left-symmetric algebras of dimension less than or equal to $4$.
\end{abstract}


\maketitle

\thispagestyle{empty}
\setcounter{tocdepth}{2}

\section{Introduction}
All algebras and all vector spaces considered in this paper are finite-dimensional algebras (resp. vector spaces)  over a commutative field $\mathbb{K}$ of characteristic $0$. Some results will be established when $\mathbb{K} = \mathbb{R}$ or $\mathbb{K} = \mathbb{C}$ or $\mathbb{K}$ is algebraically closed. Whenever necessary, to avoid ambiguity, we will specify K. Note that some results in this paper remain true in the case of infinite-dimensional algebras.

Before presenting the main results of this paper concerning pre-symplectic left-symmetric algebras, we briefly recall the known framework for flat pseudo-Euclidean Lie algebras and symplectic Lie algebras.

\subsection{The symmetric setting}

A pseudo-Riemannian Lie group is a Lie group $G$ endowed with a left-invariant pseudo-Riemannian metric $\mu$. Its Lie algebra $\mathfrak{g} = T_eG$, equipped with the bilinear form $\langle \cdot, \cdot \rangle = \mu_e$, is called a pseudo-Riemannian (or pseudo-Euclidean) Lie algebra. It is well known that there exists a unique torsion-free connection compatible with the metric, called the \emph{Levi-Civita connection}. This connection induces a bilinear product on $\mathfrak{g}$, called the \emph{Levi-Civita product}, which is defined by the Koszul formula:
\[
2\langle u \bullet v, w \rangle
= \langle [u,v], w \rangle
- \langle [v,w], u \rangle
+ \langle [w,u], v \rangle,
\qquad \text{for all } u,v,w \in \mathfrak{g}.
\]
The Levi-Civita product satisfies
\[
u \bullet v - v \bullet u = [u,v], \qquad
\langle u \bullet v, w \rangle + \langle v, u \bullet w \rangle = 0,
\qquad \text{for all } u,v,w \in \mathfrak{g}.
\]
In the framework of pseudo-Euclidean non-associative (super)algebras, this construction has been extended in \cite{ABE, BBE}. 

The curvature tensor of the Levi-Civita connection is determined
at the identity by
\[
\mathcal{K}(u,v) := \Ll_{[u,v]} - [\Ll_u, \Ll_v],
\]
where \(\Ll_u\) denotes the left multiplication operator defined by
$
\Ll_u(v) = u \bullet v,$ for all $u,v \in \mathfrak{g}.$

A pseudo-Riemannian Lie group $(G,\mu)$ is said to be \emph{flat} if its curvature $\mathcal{K}$ vanishes identically. In this case, the associated Lie algebra $(\mathfrak{g}, \br, \prs)$ is called a \emph{flat pseudo-Riemannian \textup{(}or pseudo-Euclidean\textup{)} Lie algebra}. The vanishing of the curvature is equivalent to the fact that $(\mathfrak{g}, \bullet)$ is a left-symmetric algebra. Consequently, a flat pseudo-Riemannian Lie algebra can be viewed as a left-symmetric algebra endowed with a symmetric nondegenerate bilinear form for which all left multiplication operators are skew-symmetric with respect to this form.

In the case where the metric is Riemannian (or Euclidean), i.e., positive or negative definite, J.~Milnor \cite{M} proved that a Riemannian Lie algebra $(\mathfrak{g}, \br, \prs)$ is flat if and only if $\mathfrak{g}$ decomposes  as $\mathfrak{g} = \mathfrak{b} \oplus \mathfrak{u}$, where $\mathfrak{u}$ is an abelian ideal, $\mathfrak{b}$ is an abelian subalgebra, and $\mathrm{ad}_u$ is skew-symmetric with respect to $\prs$ for all $u \in \mathfrak{b}$.

For arbitrary signatures, several works have investigated flat structures. In the Lorentzian case  ( signature $(1,n-1)$), the authors in \cite{Medina2} introduced a construction method called the \emph{double extension} for flat pseudo-Euclidean Lie algebras, which is an analogue of the double extension procedure for quadratic Lie algebras developed by A.~Medina and P.~Revoy \cite{MR}. They proved that flat Lorentzian nilpotent Lie algebras can be obtained via this procedure. Later, M.~Ait Ben Haddou, M.~Boucetta, and H.~Lebzioui \cite{ABL} further investigated flat Lorentzian Lie algebras, proving that those with a degenerate center can be obtained by the double extension procedure. In addition, it was shown in \cite{BL1} that non-unimodular flat Lorentzian Lie algebras can also be constructed using this approach. A complete classification of flat Lorentzian nilpotent Lie algebras was later obtained in \cite{Bajo}, and a full description of all flat Lorentzian Lie algebras was provided in \cite{Boucetta}.

For signature $(2,n-2)$, M.~Boucetta and H.~Lebzioui \cite{BL2} studied flat pseudo-Euclidean Lie algebras and showed that every nilpotent flat pseudo-Euclidean Lie algebra of this signature arises as a double extension of a nilpotent flat Lorentzian Lie algebra. Moreover, in \cite{LH}, the author proved that any unimodular flat pseudo-Euclidean Lie algebra of signature $(2,n-2)$ with a degenerate center arises as a double extension of a unimodular flat Lorentzian Lie algebra.

In a related direction, H.~Lebzioui \cite{LH1} studied the class $\mathcal{C}$ of flat pseudo-Euclidean Lie algebras $(\mathfrak{g}, \br, \prs)$ whose derived ideal $[\mathfrak{g},\mathfrak{g}]$ is contained in the left annihilator of the Levi-Civita product, a condition equivalent to $(\mathfrak{g},\bullet, \prs)$ being a pseudo-Euclidean left-symmetric $\Ll$-algebra. It was shown that any Lorentzian algebra in $\mathcal{C}$ with a degenerate derived ideal can be obtained via a flat double extension. This study was further extended in \cite{BO}, providing an inductive description of all elements of $\mathcal{C}$ in arbitrary signatures through double extensions by $1$ and $2$ dimensional algebras. Furthermore, pseudo-Euclidean Novikov algebras with symmetric left multiplications were investigated in the Lorentzian case in \cite{LH2}, and generalized to arbitrary signatures in \cite{BO1, BEL}, where they were shown to form a subclass of pseudo-Euclidean left-symmetric $\Ll$-algebras.

We shall now describe the situation for symplectic structures on Lie algebras.
\subsection{The skew-symmetric setting}

A \emph{symplectic Lie group} is a Lie group $G$ endowed with a left-invariant symplectic form $\Omega$. The associated Lie algebra $\mathfrak{g}=T_eG$, equipped with $\omega=\Omega_e$, is called a \emph{symplectic Lie algebra}. A connection $\nabla$ on $G$ is a \emph{symplectic connection} if it is torsion-free and compatible with $\Omega$, i.e., $\nabla \Omega = 0$. Such a connection always exists, although  it is not unique (see \cite{BCGRS}). The systematic study of symplectic Lie groups and algebras was pioneered by Bon-Yao Chu \cite{ch}, Lichnerowicz, Medina, and Revoy, as well as Dardié, Medina, and Fischer \cite{MRE, DM, LM, F}.

In \cite{BB}, a \emph{natural} symplectic connection on symplectic Lie groups was introduced, defined by:
\[
\omega(u\star v, w) = \tfrac{1}{3}\,\omega([u,v],w) + \tfrac{1}{3}\,\omega([u,w],v), \qquad \text{for all } u,v,w\in\mathfrak{g}.
\]
Analogously to the pseudo-Riemannian setting, one defines the curvature tensor associated with this connection. A symplectic Lie group $(G,\Omega)$ is said to be \emph{flat} if its curvature vanishes identically. In this case, the underlying Lie algebra $(\mathfrak{g}, \br, \omega)$ is called a \emph{flat symplectic Lie algebra}. Flatness is equivalent to the condition that the product $\star$ endows $\mathfrak{g}$ with a left-symmetric algebra structure. This framework was further extended to the superalgebraic setting in \cite{BE}, where flat symplectic Lie superalgebras were investigated. Despite its importance, the flat symplectic case has received little attention from experts. General results on flat quasi-Frobenius Lie algebras appear in \cite{V}, while the symplectic connection above was studied in \cite{BELL}.

Recall that on any symplectic Lie group $(G,\Omega)$, one can also consider the left-invariant connection $\nabla$ defined by:
\begin{equation}\label{eq:flat-conn}
\omega(\nabla_u v, w) = -\omega(v,[u,w]), \qquad \text{for all } u,v,w\in\mathfrak{g}.
\end{equation}
In \cite{Wadia}, the authors investigated the specific case where this induced left-symmetric product satisfies the Novikov identity, proving that the left-symmetric product associated with a symplectic Lie algebra is Novikov if and only if it is associative.

Furthermore, a \emph{T-symplectic Lie algebra} $(\mathfrak{g}, \br, \omega)$ is a Lie algebra $(\mathfrak{g}, \br)$ endowed with a skew-symmetric nondegenerate bilinear form $\omega$ (which is not necessarily a $2$-cocycle). In \cite{BE1}, the authors introduced the notion of the Levi-Civita product associated with a T-symplectic Lie algebra $(\mathfrak{g}, \br, \omega)$, which is analogous to the Levi-Civita product of a pseudo-Euclidean Lie algebra. This product $\bullet$ is the unique bilinear map on the underlying vector space $\g$ that satisfies the following properties:
\[
[u,v] = u\bullet v - v\bullet u, \qquad \omega(u\bullet v,w) = \omega(v, u\bullet w),
\]
for all $u,v,w\in\g$. The product $\bullet$ is also characterized by the following Koszul formula:
\[
2\omega(u\bullet v, w) = \omega([u,v], w) - \omega([v,w], u) + \omega([w,u], v).
\]

\subsection{Outline of the paper}

The paper is organized as follows:

We introduce the notion of pre-symplectic left-symmetric algebras, which are left-symmetric algebras endowed with a nondegenerate skew-symmetric bilinear form $\omega$ such that all left multiplication operators are symmetric with respect to $\omega$. In this case, the subadjacent algebra $(\mathfrak{g}^-, \omega)$ becomes a $T$-symplectic Lie algebra, and by the uniqueness of the Levi-Civita product, the left-symmetric product $\bullet$ coincides with the Levi-Civita product associated with $(\mathfrak{g}^-, \omega)$. Consequently, $(\mathfrak{g}^-, \omega)$ is a flat $T$-symplectic Lie algebra. Thus, the study of flat $T$-symplectic Lie algebras is equivalent to investigating left-symmetric algebras equipped with a nondegenerate skew-symmetric bilinear form satisfying $\omega(u \bullet v, w) - \omega(v, u \bullet w) = 0$ for all $u, v, w \in \mathfrak{g}$.  

In section 2, we introduce a distinguished subclass called \emph{Milnor} pre-symplectic algebras (see Definition~\ref{defmilnor} and Proposition~\ref{milnor}). In particular, we prove that any pre-symplectic left-symmetric algebra whose commutator ideal is nondegenerate is necessarily a Milnor pre-symplectic algebra (see Theorem~\ref{  ideal non degenere}). Furthermore, we show that if $(\mathcal{A}, \bullet, \omega)$ is a pre-symplectic left-symmetric algebra such that $[\mathcal{A}, \mathcal{A}]$ is nondegenerate, then $(\mathcal{A}, \bullet)$ is a Novikov algebra (see Corollary~\ref{NV}). Under the same nondegeneracy condition, we establish that $\mathcal{A}^{-}$ is nilpotent if and only if the product $\bullet$ is trivial (see Corollary~\ref{nilpotent}).

We then introduce the notion of representations of pre-symplectic left-symmetric algebras. More precisely, given such an algebra $(\mathcal{A}, \bullet, \omega)$, we define a new bilinear product $\star$ on $\mathcal{A}$ by means of the relation $\omega(u \bullet v, w) = \omega(u, v \star w)$ for all $u, v, w \in \mathcal{A}$. This allows us to establish a characterization of pre-symplectic left-symmetric algebras in terms of $\mathcal{A}-bimodule$ isomorphisms, proving that the existence of a suitable bilinear form is equivalent to the isomorphism between certain naturally associated $\mathcal{A}-bimodules$ (see Theorem~\ref{Thm7.11}). Finally, motivated by this representation and inspired by the $T^*$-extension for quadratic algebras introduced by M.~Bordemann \cite{Bordemann}, we provide a general construction method for pre-symplectic left-symmetric algebras (see Theorem~\ref{construction}).

In Section 3, we investigate the Levi-Civita product associated with symplectic Lie algebras. We show that the Levi-Civita product associated with a symplectic Lie algebra is a right-symmetric algebra, and we show that it forms a left-symmetric algebra if and only if it is an associative algebra. We also provide a characterization of symplectic Lie algebras in terms of representations of left-symmetric algebras, and conclude by establishing a construction method for these structures known as the $T^*$-extension (see Theorem~\ref{constriction12})

In Sections 4 and 5, we introduce the notion of double extensions of pre-symplectic left-symmetric algebras by commutative associative algebras (see Theorem~\ref{Double-extension}). In particular, we provide the explicit construction of double extensions of pre-symplectic left-symmetric algebras by one- and two-dimensional commutative associative algebras (see Theorems~\ref{double-ex1} and~\ref{double-extension2}).

In Section 6, we prove that any pre-symplectic left-symmetric algebra whose commutator ideal is degenerate can be obtained via a sequence of double extensions of a pre-symplectic left-symmetric algebra by a commutative associative algebra (see Theorem~\ref{Double-generale}). We then show that any pre-symplectic left-symmetric algebra is either a Milnor pre-symplectic algebra or can be obtained by a finite sequence of double extensions by a commutative associative algebra starting from a Milnor pre-symplectic algebra (see Theorem~\ref{Milnor0}).  In particular, we show that any pre-symplectic left-symmetric algebra whose commutator ideal is degenerate can be obtained via a sequence of double extensions by a one- or two-dimensional commutative associative algebra over the field $\mathbb{R}$ (see Theorem~\ref{db-1-2}), or by a one-dimensional algebra over an algebraically closed field (see Theorem~\ref{Closed}). Furthermore, we prove that any pre-symplectic left-symmetric algebra over $\mathbb{R}$ (resp., an algebraically closed field) is either a Milnor pre-symplectic algebra or can be obtained by a finite sequence of double extensions by a one- or two-dimensional commutative associative algebra (resp., a one-dimensional algebra) starting from a Milnor pre-symplectic algebra (see Theorem~\ref{Milnor1} and Theorem~\ref{Milnor10}, respectively).


In Section 7, as an application of the double extension procedure, we provide the complete classification of pre-symplectic left-symmetric algebras of dimension less than or equal to $4$ (see Theorem~\ref{classif4}).

  A \emph{symplectic vector space} $(E, \omega)$ is a  vector space $E$ endowed with a nondegenerate, skew-symmetric bilinear form $\omega$. A \emph{symplectic basis} of $(E, \omega)$ is a basis $\{e_1, \dots, e_{2n}\}$ of $E$ such that 
\[
\omega(e_{2i-1}, e_{2i}) = -\omega(e_{2i}, e_{2i-1}) = 1, \qquad \text{for } i = 1, \dots, n,
\]
with all other pairings between basis elements being zero.  For elements \(x,y\in E\) we denote by \(x^*,y^*\in E^*\) their duals. For every \(x^*\otimes y^*\in E^*\otimes E^*\) and  \(u\otimes v\in E\otimes E\), we define the pairing
\[
\langle x^*\otimes y^*, u \otimes v\rangle :=x^*(u)\,y^*(v).
\]
The wedge is given by
\[
x^*\wedge y^*=x^*\otimes y^* -y^*\otimes x^*.
\]
Let $F$ be a subspace of $(E, \omega)$. The orthogonal complement of $F$ with respect to $\omega$ is the subspace defined by $F^{\perp} = \{x \in E \mid \omega(x, y) = 0 \text{ for all } y \in F\}$. The subspace $F$ is said to be \emph{degenerate} (resp., \emph{nondegenerate}, \emph{totally isotropic}, \emph{Lagrangian}) if $F \cap F^{\perp} \neq \{0\}$ (resp., $F \cap F^{\perp} = \{0\}$, $F \subset F^{\perp}$, $F = F^{\perp}$). Given an endomorphism $f : E \longrightarrow E$, its adjoint $f^*$ with respect to $\omega$ is the endomorphism of $E$ uniquely defined by $\omega(f(x), y) = \omega(x, f^*(y))$ for all $x, y \in E$.

\section{Pre-symplectic left-symmetric algebras where the commutator ideal is nondegenerate}

In this section, we introduce the notion of pre-symplectic left-symmetric algebras. We then study their fundamental properties and give a characterization of pre-symplectic left-symmetric algebras whose commutator ideal is nondegenerate.

Let $(\A, \bullet)$ be a non-associative algebra.  
We denote by $\Ll^\bullet, \Rr^\bullet : \A \to \mathrm{End}(\A)$ the left and right multiplication operators, respectively.  
For any $u \in \A$, these operators are defined by
\[
\Ll^\bullet_u(v) := u \bullet v, \qquad 
\Rr^\bullet_u(v) := v \bullet u, \qquad \text{for all }\, v \in \A.
\]
On the underlying vector space $\A$, we
define the corresponding commutator $\br_\bullet$ of the product $\bullet$ by
$$[u, v]_\bullet :=u\bullet v -v\bullet u,\; \text{ for all  }  u,v\in\A.$$ 
The algebra $(\A, [\; , \; ]_\bullet)$ will be denoted by  $\A^-$. The algebra $(\A, \bullet)$ is called Lie-admissible
algebra, if $\A^-$ is a Lie algebra.  In this case,  $\A^-$ is called the sub-adjacent Lie algebra of $(\A, \bullet)$ and the latter is called a compatible (Lie-admissible) algebra structure on
the Lie algebra $\A^-$.\\

Let us recall some basic definitions. Let $(\mathcal{A},\bullet)$ be a nonassociative algebra. The \emph{associator} is the trilinear map $(\cdot, \cdot, \cdot): \mathcal{A} \times \mathcal{A} \times \mathcal{A} \to \mathcal{A}$ defined by
\[
(u,v,w) = (u \bullet v)\bullet w - u \bullet (v \bullet w).
\]

\begin{enumerate}
    \item $(\mathcal{A},\bullet)$ is \emph{associative} if its associator vanishes identically, i.e., 
    \[
    (u,v,w)=0, \quad \text{ for all } u,v,w \in \mathcal{A}.
    \]
    
    \item $(\mathcal{A},\bullet)$ is a \emph{left-symmetric algebra} (resp., a \emph{right-symmetric algebra}) if 
    \[
    (u,v,w)=(v,u,w) \quad \text{\Big(resp., } (u,v,w)=(u,w,v)\text{\Big)}, \quad \text{ for all } u,v,w \in \mathcal{A}.
    \]
    Left-symmetric algebras are well-known to be Lie-admissible.

   

    \item $(\mathcal{A},\bullet)$ is an \emph{$\mathrm{L}$-algebra} if, for any $u,v,w\in\mathcal{A}$,
    \[
    u\bullet(v\bullet w) = v\bullet(u\bullet w),
    \]
    which means $\Ll^\bullet_u\circ \Ll^\bullet_v=\Ll^\bullet_v\circ \Ll^\bullet_u$.

    \item $(\mathcal{A},\bullet)$ is an \emph{$\mathrm{R}$-algebra} if, for any $u,v,w\in\mathcal{A}$,
    \[
    (u\bullet v)\bullet w = (u\bullet w)\bullet v,
    \]
    which means $\Rr^\bullet_w\circ \Rr^\bullet_v=\Rr^\bullet_v\circ \Rr^\bullet_w$.

   \item $(\mathcal{A},\bullet)$ is an \emph{$\mathrm{LR}$-algebra} if it is both an $\mathrm{L}$-algebra and an $\mathrm{R}$-algebra. In this case, $[\cdot,\cdot]_\bullet$ defines a Lie bracket (see \cite[Lemma 1.3]{BDD}). An \emph{$\mathrm{LR}$-structure} on a Lie algebra $\mathfrak{g}$ is a product $\bullet$ such that $(\mathfrak{g}, \bullet)$ is an $\mathrm{LR}$-algebra whose associated Lie algebra $\mathfrak{g}^-$ coincides with $\mathfrak{g}$. These structures arise naturally in studying affine actions on nilpotent Lie groups (see \cite{BDV}).

    \item $(\mathcal{A},\bullet)$ is a \emph{Novikov algebra} if it is both a left-symmetric algebra and an $\mathrm{R}$-algebra.

    \item $(\mathcal{A},\bullet)$ is a \emph{left-Leibniz algebra} if, for all $u,v,w\in\mathcal{A}$,
    \[
    (u\bullet v)\bullet w=u\bullet (v\bullet w)-v\bullet (u\bullet w),
    \]
    which is equivalent to $\Ll^\bullet_{u\bullet v}=[\Ll^\bullet_u,\Ll^\bullet_v]$. 
\end{enumerate}

\ssbegin{Definition}
We say that $(\mathfrak{g}, \br, \omega)$ is a \emph{$T$-symplectic Lie algebra} if $(\mathfrak{g}, \br)$ is a Lie algebra equipped with a nondegenerate and skew-symmetric bilinear form $\omega$ on $\mathfrak{g}$.
\end{Definition}

In \cite{BE1}, the notion of the Levi-Civita product was introduced for anticommutative nonassociative algebras equipped with a nondegenerate skew-symmetric bilinear form. This notion was later extended to anticommutative superalgebras in \cite{BE}. In the present work, we focus exclusively on the Levi-Civita product associated with $T$-symplectic Lie algebras, as recalled in the following proposition.

\ssbegin{Proposition}[\cite{BE1}]\label{pdoduit de Levi-civita}
Let $(\mathfrak{g},[\cdot,\cdot], \om)$ be a $T$-symplectic Lie algebra. Then there exists a unique product $\bullet$ on $\mathfrak{g}$ satisfying
\begin{align}
    u\bullet v - v\bullet u &= [u,v], \label{torsion} \\
    \omega(u\bullet v,w) - \omega(v,u\bullet w) &= 0, \label{compatible}
\end{align}
for all $u,v,w \in \mathfrak{g}$. More precisely, this product is defined by
\begin{equation}
\omega(u\bullet v,w) = \frac{1}{2} \Big( \omega([u,v],w) - \omega([v,w],u) + \omega([w,u],v) \Big). \label{koszul}
\end{equation}
The product $\bullet$ is called the \emph{Levi-Civita product} associated with the $T$-symplectic Lie algebra $(\mathfrak{g},[\cdot,\cdot],\omega)$.
\end{Proposition}

According to relations \eqref{torsion} and \eqref{compatible}, for any $u \in \mathfrak{g}$, the left multiplication operator $\Ll_u^\bullet$ is symmetric with respect to $\omega$, and the adjoint operator satisfies 
$
\mathrm{ad}_u = \Ll_u^\bullet - \Rr_u^\bullet,$ where $\mathrm{ad}_u: \mathfrak{g} \to \mathfrak{g}$ is defined by $\mathrm{ad}_u(v) = [u,v]$ for all $v\in \mathfrak{g}$.

The \emph{curvature operator} on $(\mathfrak{g}, \omega)$ is defined by
\[
\mathcal{K}(u,v) := \Ll^\bullet_{[u,v]} - [\Ll^\bullet_u, \Ll^\bullet_v], \quad \text{for all } u,v \in \mathfrak{g}.
\]

If this operator vanishes identically, then $(\mathfrak{g},[\cdot,\cdot], \omega)$ is called a \emph{flat $T$-symplectic Lie algebra}. 

Clearly, the curvature $\mathcal{K}$ vanishes identically if and only if $(\mathfrak{g}, \bullet)$ is a left-symmetric algebra.

The problem we study can be approached in two equivalent ways.

\medskip
\noindent \textbf{Approach 1.}
Let $(\mathfrak{g},[\cdot,\cdot],\omega)$ be a $T$-symplectic Lie algebra and let $\bullet$  denote its Levi-Civita
product.  We investigate conditions under which $(\mathfrak{g},[\cdot,\cdot],\omega)$ is flat. The flatness of this algebra implies that $(\mathfrak{g},\bullet)$ is left-symmetric, where the nondegenerate skew-symmetric bilinear form $\omega$ satisfies $\omega(u\bullet v, w) = \omega(v, u\bullet w)$ for all $u,v,w\in \mathfrak{g}$.

\medskip
\noindent \textbf{Approach 2 (Conversely).}
Let $(\mathfrak{g},\bullet)$ be a left-symmetric algebra equipped with a nondegenerate skew-symmetric bilinear form $\omega$ such that
\[
\omega(u\bullet v,w) - \omega(v, u\bullet w) = 0, \quad \text{ for all }u,v,w\in\mathfrak{g}.
\]
Then, $(\mathfrak{g^{-}},\omega)$ is a flat $T$-symplectic Lie algebra with the Lie bracket defined by $[u,v]_\bullet := u\bullet v - v\bullet u$. Furthermore, by the uniqueness of the Levi-Civita product, $\bullet$ coincides with the Levi-Civita product associated with $(\mathfrak{g^{-}},\omega)$.

\medskip
\noindent 
We conclude that the study of flat $T$-symplectic Lie algebras is equivalent to the study of left-symmetric algebras equipped with a nondegenerate skew-symmetric bilinear form satisfying
\[
\omega(u\bullet v , w) - \omega(v , u\bullet w) = 0, \qquad \text{ for all }u,v,w\in\mathfrak{g}.
\]

This motivates the following definition.

\ssbegin{Definition}\label{Defpr}
A \emph{pre-symplectic algebra} is a triple
$(\mathcal A, \bullet, \om)$, where
$(\mathcal A, \bullet)$ is an algebra  and
$\om$ is a nondegenerate skew-symmetric bilinear form on
$\mathcal A$ satisfying 
\[
\om( u\bullet v , w) - \om( v , u\bullet w) =0,
\qquad \text{ for all } u,v,w\in\A .
\]
\end{Definition}

\ssbegin{Proposition}\label{condition plate}
Let $(\mathcal{A},\bullet)$ be an algebra equipped with a nondegenerate skew-symmetric bilinear form $\omega$ such that the left multiplication operators are symmetric with respect to $\omega$. Then, $(\mathcal{A},\bullet)$ is a left-symmetric algebra if and only if
\begin{equation}\label{plat}
\Ll^\bullet_{[u,v]_\bullet}=0 \quad \text{and} \quad \Ll^\bullet_u\circ \Ll^\bullet_v=\Ll^\bullet_v\circ \Ll^\bullet_u, \qquad \text{for all } u,v\in \mathcal{A}.
\end{equation}
\end{Proposition}

\begin{proof}
Assume that $(\mathcal{A},\bullet)$ is a left-symmetric algebra. Thus, for all $u,v\in \mathcal{A}$, we have $$\Ll^\bullet_{[u,v]_\bullet}=[\Ll^\bullet_u, \Ll^\bullet_v].$$ Since the operator $\Ll^\bullet_{[u,v]_\bullet}$ is symmetric with respect to $\omega$, whereas the commutator $[\Ll^\bullet_u, \Ll^\bullet_v]$ is skew-symmetric with respect to $\omega$, both sides must vanish identically. It follows that
\[
\Ll^\bullet_{[u,v]_\bullet} = 0 \quad \text{and} \quad [\Ll^\bullet_u, \Ll^\bullet_v] = 0, \qquad \text{for all } u,v\in \mathcal{A}.
\]
Hence, the identities \eqref{plat} hold.

Conversely,  the result follows immediately.
\end{proof}


\ssbegin{Remark}
Under the assumptions of Proposition~\ref{condition plate},
a left-symmetric algebra $(\mathcal{A},\bullet)$
satisfying the $\mathrm{L}$-algebra identities
will be called a left-symmetric $\mathrm{L}$-algebra.
\end{Remark}

Let $(\mathcal{A},\bullet,\omega)$ be a pre-symplectic left-symmetric algebra. Denote by
\[
[\mathcal{A},\mathcal{A}] :=\spa \{[u,v]_\bullet \mid u,v\in\mathcal{A}\}
\]
the commutator ideal of the associated Lie algebra $\mathcal{A}^-$, and by
\[
\mathcal{A}\bullet\mathcal{A} :=\spa \{u\bullet v \mid u,v\in\mathcal{A}\}
\]
the two-sided ideal of $(\mathcal{A},\bullet)$. We obviously have
\begin{equation}\label{ortogonal}
(\mathcal{A}\bullet\mathcal{A})^\perp =\{u\in\mathcal{A} \mid \Rr^\bullet_u=0\} \qquad \text{and} \qquad [\mathcal{A},\mathcal{A}]^\perp = \{u\in\mathcal{A} \mid (\Rr_u^\bullet)^*=-\Rr^\bullet_u\}.
\end{equation}

\ssbegin{Proposition}\label{pr 6.7}
 Let $(\A, \bullet, \om)$ be a pre-symplectic left-symmetric algebra. Then:
 \begin{enumerate}
 \item[$(i)$] $Z(\A^-) \cap[\A, \A]$ is contained in $[\A, \A]^\bot$. Consequently, $Z(\A^-) \cap[\A, \A]$ is a completely isotropic two-sided ideal of $(\A, \bullet)$.
\item[$(ii)$] $[\A, \A]$ is a two-sided ideal of $(\A, \bullet)$.
\item[$(iii)$] $[\A, \A]^\bot$ is an anti-commutative left ideal of $(\A, \bullet)$. Moreover, $u\bullet (v\bullet w) =(u\bullet v)\bullet w=0$ for all $u,v,w\in [\A, \A]^\bot$.
\item[$(iv)$] $[\A,\A]^\bot\cap[\A,\A]$ and $[\A,\A]+[\A,\A]^\bot$ are  two-sided ideals of $(\A,\bullet)$.

\end{enumerate}
\end{Proposition}

 \begin{proof}
$(i)$. According to Proposition \ref{condition plate}, we have
$
[u,v]_\bullet \bullet w =0,$ for all $ u,v,w\in \A.$ Then, by  \eqref{ortogonal}, it follows that
$Z(\A^-) \cap[\A,\A] \subseteq (\A\bullet \A)^\perp.$  Moreover, thanks to  $[\A,\A]\subset (\A\bullet \A)$, one has $
Z(\A^-) \cap[\A,\A] \subseteq (\A\bullet \A)^\perp \cap (\A\bullet \A).$ Therefore, $Z(\A^-) \cap[\A, \A]$ is a completely isotropic two-sided ideal of $(\A, \bullet)$.

 $(ii)$. According to Proposition \ref{condition plate}, we have $[u,v]_\bullet\bullet w=0$ for all $u,v,w\in \A$ and $w\bullet [u,v]_\bullet=[w, [u,v]_\bullet]_\bullet\in [\A,\A]$, we get that $[\A,\A]$ is a two-sided ideal of $(\A, \bullet)$.

$(iii)$.  Let $u\in \A$, $v\in[\A,\A]^\bot$ and $w\in [\A,\A]$ we have $\om (u\bullet v, w)\overset{\eqref{ortogonal}}{=}\om(u, w\bullet v)=0$, then $[\A,\A]^\bot$ is a left ideal of $(\A, \bullet)$. On the other hand, let $u,v\in[\A,\A]^\bot$. For all $w\in\A$ we have 
$$\om(u\bullet v,w)\overset{\eqref{ortogonal}}{=}-\om( u, w\bullet v)= -\om( w\bullet u,  v)\overset{\eqref{ortogonal}}{=}\om( w,  v \bullet u)=- \om(   v \bullet u, w).$$
As $\om$ is nondegenerate, it follows that $u\bullet v=-v\bullet u.$ Let $u, v,w\in [\A, \A]^\bot$. Since $(\A, \bullet)$ is  a $\Ll$-algebra and $[\A, \A]^\bot$ is an anti-commutative subalgebra of $(\A, \bullet)$, then $(u\bullet v-v\bullet u)\bullet w=0$ and hence $2(u\bullet v)\bullet w=0$  and $(u\bullet v)\bullet w=w\bullet (u\bullet v) =0.$

 (iv) Since $[\mathcal{A},\mathcal{A}]$ and $[\mathcal{A},\mathcal{A}]^\perp$ are left ideals of $(\mathcal{A}, \bullet)$, it follows that $[\mathcal{A},\mathcal{A}]^\perp \cap [\mathcal{A},\mathcal{A}]$ is also a left ideal of $(\mathcal{A}, \bullet)$. Moreover, since 
$
([\mathcal{A},\mathcal{A}]^\perp \cap [\mathcal{A},\mathcal{A}]) \bullet \mathcal{A} = \{0\},
$
we conclude that $[\mathcal{A},\mathcal{A}]^\perp \cap [\mathcal{A},\mathcal{A}]$ is a right ideal of $(\mathcal{A}, \bullet)$.

On the other hand, let $u \in [\mathcal{A},\mathcal{A}]^\perp \cap [\mathcal{A},\mathcal{A}]$, $v \in \mathcal{A}$, $w_1 \in [\mathcal{A},\mathcal{A}]$, and $w_2 \in [\mathcal{A},\mathcal{A}]^\perp$. By $(ii)$ and $(iii)$, we obtain
\[
\omega\big((w_1+w_2)\bullet v, u\big) = \omega(w_2\bullet v, u) = \omega(v, w_2\bullet u) = -\omega(v, u\bullet w_2) = 0.
\]
Since $\omega$ is nondegenerate,
$[\mathcal A,\mathcal A]+
[\mathcal A,\mathcal A]^\perp$
is a right ideal of $(\mathcal A,\bullet)$.

In a similar way, we obtain
\[
\omega\big(v\bullet(w_1+w_2), u\big) = \omega(w_1+w_2, v\bullet u) = 0,
\]
which shows that $[\mathcal{A},\mathcal{A}] + [\mathcal{A},\mathcal{A}]^\perp$ is also a left ideal of $(\A, \bullet)$.
\end{proof}

\ssbegin{Lemma}\label{Lemma 1}
	Let $(\A,\bullet)$ be an algebra equipped with a nondegenerate skew-symmetric bilinear form $\omega$ such that the left multiplication operators are symmetric with respect to $\omega$.
  If
\(
[\Rr^\bullet_u,\Rr^\bullet_v]=0\)
 for all $u,v\in \mathcal A,$
then
\[
(u\bullet v)\bullet w = 0,
\quad \text{for all } u,v,w\in \mathcal A.
\]
Equivalently,
\[
\Ll^\bullet_{u\bullet v} = 0
\quad\text{for all } u,v\in \mathcal A,
\quad\text{or equivalently}\quad
\Rr^\bullet_u \circ \Rr^\bullet_v = 0
\quad\text{for all } u,v\in \mathcal A.
\]
\end{Lemma}
\begin{proof}
	 For any $u,v,w,z\in \A$, we have 
	\begin{eqnarray*}
		\om( \Rr^\bullet_u\circ\Rr^\bullet_v(w), z ) &=& \om( (w\bullet v)\bullet
		u,z)= \om( u,(w\bullet v)\bullet z ),\\
		&=&\om( u,(w\bullet z)\bullet v )=\om (w\bullet z)\bullet u,v ),
		\\&=&\om( (w\bullet u)\bullet z,v) =\om( z,(w\bullet u)\bullet v ),\\ &=&-\om( \Rr^\bullet_v\circ\Rr^\bullet_u(w), z ).
	\end{eqnarray*}
	As $\om$  is nondegenerate,  we deduce that $\Rr^\bullet_u\circ\Rr^\bullet_v=0$. 
	\end{proof}

\ssbegin{Proposition}\label{Condition Novikov}
	Let $(\A,\bullet,\omega)$ be an algebra equipped with a nondegenerate skew-symmetric bilinear form $\omega$ such that the left multiplication operators are symmetric with respect to $\omega$. Then the following assertions are equivalent:
	\begin{enumerate}\item[$(i)$] 
	 $(\A, \bullet)$ is a  Novikov  algebra.
	 \item[$(ii)$] $(\A, \bullet)$ is an $\mathrm{LR}$-algebra.
 \item[$(iii)$] $(\A,\bullet, \om)$ is a  left Leibniz $\Ll$-algebra.
  \item[$(iv)$]  For any $u,v\in\A$,  
	\begin{equation}
		\Ll^\bullet_{u\bullet v }=[\Ll^\bullet_u,\Ll^\bullet_v]=0.   \label{kly0}
	\end{equation}
	\end{enumerate}
	
\end{Proposition}
\begin{proof}  It is an immediate consequence of Lemma \ref{Lemma 1}.
\end{proof}

\ssbegin{Proposition}\label{[g,g] abelian}
	Let $(\A, \bullet, \om)$ be a pre-symplectic left-symmetric algebra. Then 
	\begin{enumerate}
		\item[$(i)$]  The restriction of $\bullet$ to $[\A,\A]$ is trivial; in particular, $[\A,\A]$ is an abelian ideal of $\A^-$.
		\item[$(ii)$]  If $[\A, \A]$ is nondegenerate, then $ [\A, \A]=\A\bullet\A$ and  $ [\A, \A]^{\perp}$ is an abelian Lie subalgebra of $\A^-$.
	\end{enumerate}

\end{Proposition}
\begin{proof}

$(i)$. This is an immediate consequence of \eqref{plat}.

$(ii)$.  Suppose that $[\A,\A]$ is nondegenerate.  By Proposition~\ref{pr 6.7}, $[\A, \A]^\bot$ is an anti-commutative left ideal. Hence
\[
[\A, \A]^\bot\bullet [\A, \A]^\bot\subseteq [\A, \A]^\bot\cap [\A, \A]=\{0\}.
\]
Since $[\A, \A]\bullet\A=0$, it follows that $\A\bullet [\A, \A]^\bot=0$. Consequently,
$
[\A, \A]^\bot\subseteq(\A\bullet\A)^\perp.
$
Conversely, since $[\A, \A]\subseteq\A\bullet\A$, we have
$
(\A\bullet\A)^\perp\subseteq [\A, \A]^\bot.
$
Thus,
$
(\A\bullet\A)^\perp=[\A, \A]^\bot,
$
which implies that
$
\A\bullet\A=[\A,\A].
$

Finally, the restriction of $\bullet$ to $(\A\bullet\A)^\perp$ is trivial, and thus $[\A, \A]^\perp$ is an abelian Lie subalgebra.




\end{proof}

We now introduce an important class of pre-symplectic  left-symmetric algebras. In \cite[Theorem 1.5]{M}, Milnor showed that the Lie algebra $\mathfrak g$ of a flat Riemannian (or Euclidean) Lie algebra decomposes as a semidirect product of an abelian ideal $\mathfrak u$ and an abelian subalgebra $\mathcal B$. Furthermore, for any $u \in \mathcal B$, the adjoint map $\operatorname{ad}_u$ is antisymmetric. 

With respect to this  decomposition  $\mathfrak{g}=\mathfrak{u}\oplus\mathcal{B}$, the
Levi-Civita product $\bullet$ satisfies
$$
\Ll^\bullet_u =
\begin{cases}
0, & \text{if } u\in \mathfrak{u}, \\[2mm]
\ad_u, & \text{if } u\in \mathcal{B}.
\end{cases}
$$
Here $\ad_u=\Ll^\bullet_u-\Rr^\bullet_u$, for all $u\in\g$. This is
equivalent to the fact that $\g=\mathfrak{u}\oplus \mathfrak{u}^\bot$, with 
\[
\g\bullet\g \subseteq \mathfrak{u}
\quad \text{and} \quad
\Ll^\bullet_u = 0,
\ \text{for all } u\in\mathfrak{u}.
\]

We enlarge the context and introduce the following definition.
	
	\ssbegin{Definition}\label{defmilnor} 
A Milnor pre-symplectic algebra is a pre-symplectic algebra $(\mathcal{A},\bullet,\omega)$ such that $(\mathcal{A},\bullet)$ is a Lie-admissible algebra and there exists a \emph{nondegenerate} subspace $I$ of $\mathcal{A}$ such that $\mathcal{A}\bullet\mathcal{A}\subset I$ and $\Ll^\bullet_u=0$ for all $u\in I$. In this case, we also say that $(\mathcal{A}, \bullet, \omega)$ is a Milnor pre-symplectic algebra by means of $I$.
\end{Definition}
     The following proposition justifies this terminology and shows that Milnor  algebras form a subclass of pre-symplectic Novikov  algebras.

\ssbegin{Proposition}\label{milnor} Let $(\A,\bullet,\om)$ be a Milnor pre-symplectic algebra by means of $I$. Then:
		\begin{enumerate}
			\item[$(i)$] both $I$ and $I^\perp$ are  trivial subalgebras of $(\A,\bullet)$. Thus $I$ is an abelian ideal of $\A^-$ and $I^\perp$ is an abelian Lie subalgebra, 
			\item[$(ii)$] $\A=I\oplus I^\perp$ and, for any $u\in \A$,
			\[ \Ll^\bullet_u=\begin{cases}
				0\quad\mbox{if}\quad u\in I,\\
				\ad_u\quad\mbox{if}\quad u\in I^\perp,
			\end{cases} \]
		\item[$(iii)$]  the center $Z(\A^-)$ of $(\A,\br)$ satisfies $Z(\A^-)=\{u\in\A,\Ll^\bullet_u=\Rr^\bullet_u=0\}$.
		\end{enumerate} 
	In particular, $(\A,\bullet,\om)$ is a pre-symplectic Novikov algebra.
\end{Proposition}
\begin{proof}
			 (i).  It is a consequence of the fact that $\Ll^\bullet_u=0$, for any $u\in I$, and the fact that $I^\perp\subset(\A\bullet\A)^\perp$ which implies $\Rr^\bullet_u=0$, for any $u\in I^\perp$,  by virtue of \eqref{ortogonal}.

		 $(ii)$. The only thing to point out is that $I^\perp\subset(\A\bullet\A)^\perp$ and so, for any $u\in I^\perp$, $\Rr^\bullet_u=0$ and $\ad_u=\Ll^\bullet_u$.

 (iii). Let $v=v_1+v_2\in Z(\A^-)$ where 
		$v_1\in I$ and $v_2\in I^\perp$. Since both $I$ and $I^\perp$ are abelian Lie subalgebras, we deduce that $v_1,v_2\in Z(\A^-)$. 
From $0 = \ad_{v_1} = \Ll^\bullet_{v_1} - \Rr^\bullet_{v_1}$ and 
$0 = \ad_{v_2} = \Ll^\bullet_{v_2} - \Rr^\bullet_{v_2}$, we deduce that 
$\Ll^\bullet_{v_1} = \Rr^\bullet_{v_1}$ and $\Ll^\bullet_{v_2} = \Rr^\bullet_{v_2}$. 

Moreover, since $v_1 \in I$, we have 
$\Ll^\bullet_{v_1} = \Rr^\bullet_{v_1} = 0$, and since $v_2 \in I^\perp$, we also have 
$\Ll^\bullet_{v_2} = \Rr^\bullet_{v_2} = 0$. 

Therefore, we obtain
\[
\Ll^\bullet_v = \Ll^\bullet_{v_1} + \Ll^\bullet_{v_2} = 0, \qquad
\Rr^\bullet_v = \Rr^\bullet_{v_1} + \Rr^\bullet_{v_2} = 0.
\]

The fact that $(\A,\bullet,\om)$ is pre-symplectic  Novikov algebra is an immediate consequence of $(ii)$ and the fact that $\br_\bullet$ is a Lie bracket.  Indeed, since 
$\A \bullet \A \subseteq I$, we have $\Ll^\bullet_{u\bullet v} =0$, for all $u,v \in \A$, and 
$[\Ll^\bullet_u, \Ll^\bullet_w] \overset{(ii)}{=} 0$ for all $u \in I$ and $w \in \A$. Moreover, since $I^\bot$ is an abelian Lie subalgebra,   for $w,z \in I^\perp$, we have 
\[
0 = \ad_{[w,z]} = [\ad_w,\ad_z] = [\Ll^\bullet_w-\Rr^\bullet_w, \Ll^\bullet_z-\Rr^\bullet_z] \overset{(ii)}{=} [\Ll^\bullet_w, \Ll^\bullet_z].
\]
Hence, $(\A, \bullet)$ satisfies the identities \eqref{kly0}, and thus 
$(\A, \bullet)$ is a Novikov algebra.
\end{proof}
In the following, we will give a relationship between some classes of pre-symplectic left-symmetric algebras. 

\ssbegin{Theorem} \label{ ideal non degenere}
	Let $(\A, \bullet, \om)$ be a  pre-symplectic left-symmetric algebra such that $[\A,\A]$ is  nondegenerate. Then $(\A,\bullet,\om)$ is a Milnor pre-symplectic  algebra by means of $\A\bullet\A=[\A,\A]$ and the conclusions of Proposition \ref{milnor} hold.
\end{Theorem}
\begin{proof} 
Since $[\A,\A]$ is nondegenerate, by Proposition~\ref{[g,g] abelian} we have $ \A\bullet\A=[\A,\A]. $ Then, by Proposition~\ref{condition plate}, we obtain $\Ll^\bullet_u=0$ for all $u\in [\A, \A]$. Hence $(\A,\bullet,\om)$ is a Milnor pre-symplectic algebra with 
$\A\bullet\A=[\A,\A]$, and the conclusions of Proposition~\ref{milnor} hold.
\end{proof}

\ssbegin{Corollary}\label{NV}
Let $(\A, \bullet, \om)$ be a pre-symplectic left-symmetric algebra such that $[\A,\A]$ is nondegenerate. Then $(\A, \bullet)$ is a pre-symplectic Novikov algebra. 
\end{Corollary}

\begin{proof}
By Theorem~\ref{ ideal non degenere},
$(\mathcal A,\bullet,\omega)$ is a
Milnor pre-symplectic algebra.
The conclusion follows from
Proposition~\ref{milnor}.
\end{proof}


\ssbegin{Corollary}\label{nilpotent}
Let $(\A, \bullet, \om)$ be a pre-symplectic left-symmetric algebra such that $[\A, \A]$ is nondegenerate. Then, $\A^-$ is nilpotent if and only if  $\bullet$ is trivial. 
\end{Corollary}

\begin{proof}
 Suppose, for a contradiction, that $\A^-$ is nilpotent and that the product $\bullet$ is non-trivial. According to Theorem~\ref{  ideal non degenere}, 
$(\A,\bullet,\om)$ is a Milnor pre-symplectic algebra by means of  $\A\bullet\A=[\A,\A]$. Since $\A^-$ is nilpotent, we have
$
Z(\A^-)\cap [\A,\A]\neq 0.$ By part (iii) of Proposition~\ref{milnor} and Eq.~\eqref{ortogonal}, we obtain
$
Z(\A^-)\subseteq  [\A, \A]^{\bot}.
$
Hence \[
\{0\}\neq Z(\A^-)\cap [\A,\A]\subseteq [\A,\A]^{\bot}\cap [\A,\A],
\]
which is a contradiction since $[\A,\A]$ is nondegenerate. 

Conversely, the statement is obvious.
\end{proof}

\ssbegin{Example}\label{ex4}
Let $(\A, \bullet, \omega)$ be a  pre-symplectic non-commutative left-symmetric algebra of dimension $4$ such that $[\A,\A]$ is nondegenerate. Let $\{e_1,e_2\}$ be a symplectic basis of $[\A,\A]$ and let $\{f_1,f_2\}$ be a symplectic basis of $[\A,\A]^\perp$. By Theorem~\ref{ ideal non degenere}, we have
\[
f_1\bullet e_1=xe_1,\qquad
f_1\bullet e_2=xe_2,\qquad
f_2\bullet e_1=ye_1,\qquad
f_2\bullet e_2=ye_2,
\]
where $x,y\in\mathbb{K}$. Since $(\mathcal{A}, \bullet)$ is non-commutative, we have $(x, y) \neq (0, 0)$. Consequently, there exist $a, b \in \mathbb{K}$ satisfying $a x + b y = 1$. Consider the change of basis defined by
\[
(y_1,y_2,x_1,x_2)
=
\left(
a f_1 + b f_2,
-y f_1+x f_2,
e_1,
e_2
\right).
\]
Then $(y_1,y_2,x_1,x_2)$ is a symplectic basis of $(\A,\omega)$ and the non-zero products are given by
\[
y_1\bullet x_1=x_1,
\qquad
y_1\bullet x_2=x_2.
\]

Moreover,
\[
\omega
=
y_1^*\wedge y_2^*
+
x_1^*\wedge x_2^*.
\]

Therefore, $(\mathcal{A},\bullet,\omega)$ is isomorphic to $\big(\operatorname{span}\{ y_1,y_2, x_1,x_2\}, \bullet, \omega\big)$.
\end{Example}

\subsection{Representations of pre-symplectic left-symmetric algebras}
Here, we show, as a consequence of the results obtained so far, that every pre-symplectic left-symmetric algebra naturally carries the structure of a left-symmetric $\Ll$-algebra. We then introduce the notion of representation of a left-symmetric $\Ll$-algebra and define the product naturally associated with pre-symplectic left-symmetric algebra. Furthermore, we establish a characterization of pre-symplectic left-symmetric algebra in terms of bimodule isomorphisms. More precisely, we show that the existence of a suitable bilinear form is equivalent to the isomorphism between certain naturally associated $\A$-bimodules. Motivated by this representation, we construct a new class of pre-symplectic left-symmetric algebras, analogous to the $T^*$-extension for quadratic algebras introduced in \cite{Bordemann}.

\sssbegin{Definition}\label{Def3.1}
Let $(\A,\cdot)$ be a non-associative algebra satisfying some polynomial identities $(P_1),\dots,(P_n)$. Let $V$ be a vector space, and let 
$
l, r : \A \longrightarrow \mathrm{End}(V)
$
be two linear maps. On the direct sum $\A \oplus V$, define a product $\diamond$ by
\begin{equation}\label{eq7}
(u+x)\diamond (v+y) := u \cdot v + l(u)(y) + r(v)(x), \text{ for all $u,v \in \A$ and $x,y\in V$.}
\end{equation}

We say that $(V,l,r)$ is an $\A$-bimodule if $(\A \oplus V,\diamond)$ is an algebra satisfying the polynomial identities $(P_1),\dots,(P_n)$. In this case, the pair $(l,r)$ is called a representation of $\A$ in $V$.
\end{Definition}

\sssbegin{Proposition}\label{Prop3.3}
Let $(\A, \bullet)$ be a left-symmetric $\Ll$-algebra, $V$ be a vector space, and  
$
l, r : \A \longrightarrow \mathrm{End}(V)
$
be two linear maps. Then $(l, r)$ is a representation of $\A$ in $V$ if and only if, for all $u,v \in \A$, the following identities hold: 
\begin{equation*}\label{eq3.13}
\begin{aligned}
\relax l([u, v]_\bullet)=0,\quad [l(u),l(v)]=0, \quad r(u\bullet v)=l(u)\circ r(v),\; r(u)\circ l(v)= r(u)\circ r(v).
\end{aligned}
\end{equation*}
\end{Proposition}

\begin{proof}
By a direct computation.
\end{proof}

\sssbegin{Remark}\label{Prop3.4}
Let $(\A, \bullet)$ be a left-symmetric $\Ll$-algebra, and let 
$
\Ll^\bullet,\Rr^\bullet : \A \longrightarrow \mathrm{End}(\A)
$
be the left and right multiplication operators, respectively. Then $( \Ll^\bullet, \Rr^\bullet)$ is a representation of $\A$ in $\A$, called the \emph{adjoint representation} \textup{(}or \emph{regular representation}\textup{)} of $\A$.
\end{Remark}

Let $(\A, \bullet, \om)$ be a pre-symplectic  left-symmetric  algebra. Consider the bilinear map $\star: \A \times \A \longrightarrow \A$ defined by
\begin{equation}
\om(u\bullet v, w)=\om(u, v \star w), \quad \text{ for all } u,v,w \in \A .
\label{produce associe}\end{equation}

\sssbegin{Proposition}\label{Pr associe}
The product $\star$ defined in Eq. \eqref{produce associe} satisfies the following identities:
\begin{enumerate}
    \item[$(a)$] $u \star v = -v \star u$ for all $u,v\in\mathcal{A}$,
    \item[$(b)$] $u\star (v\star w) = u\bullet (v\star w)$ for all $u,v,w\in\mathcal{A}$,
    \item[$(c)$] $(u\bullet v)\star w = v\star (u \bullet w)$ for all $u,v,w \in \mathcal{A}$.
\end{enumerate}
Moreover, if $(\mathcal{A}, \bullet, \omega)$ is a pre-symplectic Novikov algebra, then identities $(a)$ and $(c)$ still hold, while identity $(b)$ becomes:
\begin{enumerate}
    \item[$(b')$] $u \star(v \star w) = u \bullet(v \star w) = 0$ for all $u,v,w\in\mathcal{A}$.
\end{enumerate}
\end{Proposition}

\begin{proof}
Let $u,v,w,z \in \A$.
     \begin{enumerate}
         \item[$(a)$] We have\begin{align*}
      \om( u, v\star w)=&      \om( u \bullet v, w)= \om( v, u\bullet w)= -\om( u\bullet w, v)=-\om( u, w\star v).
            \end{align*}
Thus, the product $\star$ is skew-symmetric.
\item[$(b)$] Since $(\mathcal A,\bullet)$ is a
left-symmetric $\mathrm L$-algebra,
we have
\begin{align*}
0&=\om(  (u\bullet v)\bullet w-   (v\bullet u)\bullet w, z)= \om( u, v\star (w\star z)) -\om( u, v\bullet (w\star z)) \\&=\om( u, v\star (w\star z)- v\bullet (w\star z)).
\end{align*}
Since $\om$ is nondegenerate, then \[
v\star (w\star z)= v\bullet (w\star z).
\]
\item[$(c)$]  Since $(\A, \bullet)$ is left-symmetric $\Ll$-algebra,  we have 
\begin{align*}
0&=\om(  u\bullet (v\bullet w)-   v\bullet (u\bullet w), z)= \om( u, (v\bullet w)\star z) -\om( u, w\star(v\bullet  z)) \\&=\om( u, (v\bullet w)\star z- w\star(v\bullet  z)).
\end{align*}
Since $\om$ is nondegenerate, then \[
(v\bullet w)\star z=w\star(v\bullet  z).
\]
\end{enumerate}
Now, assume that $(\mathcal{A}, \bullet, \omega)$ is a pre-symplectic Novikov algebra. According to Proposition~\ref{Condition Novikov}, we have
\begin{align*}
0\overset{\eqref{kly0}}{=}\om( (u\bullet v)\bullet w, z)=\om( u\bullet v, w\star z)= \om( u , v\star( w\star z)).
\end{align*}
Since $\omega$ is nondegenerate, we get $v\star (w\star z) = 0$. Consequently, by $(b)$, we obtain
\[
v\star( w\star z) = v\bullet (w\star z) = 0.
\]

\end{proof}

\sssbegin{Proposition}
Let $(\A, \bullet, \om)$ be a pre-symplectic Novikov algebra. Then the associated Lie algebra $(\A, \star)$ is at most $2$-step nilpotent.

Conversely, let $(\g, \star, \om)$ be a $2$-step nilpotent Lie algebra endowed with a nondegenerate skew-symmetric bilinear form. Define a bilinear product $\bullet$ on $\g$ by
\[
\om( u \bullet v, w ) = \om( u, v \star w ),
\qquad \text{ for all } u,v,w \in \g.
\]
If identity $(c)$ of Proposition~\ref{Pr associe} is satisfied, then $(\g, \bullet, \om)$ is a pre-symplectic Novikov algebra.
\end{Proposition}

\begin{proof}
The first assertion follows directly from parts (a) and (b') of Proposition~\ref{Pr associe}. 

Conversely, assuming that $({\mathfrak g}, \star)$ is a $2$-step nilpotent Lie algebra and that condition $(c)$ holds, one verifies that the Novikov identities are satisfied. The proof is straightforward and follows the same arguments as in Proposition~\ref{Pr associe}.
\end{proof}

Let $(\mathcal{A}, \bullet)$ be an algebra. For a linear map $\phi : \mathcal{A} \to \mathrm{End}(V)$, we define the dual linear map $\phi^* : \mathcal{A} \to \mathrm{End}(V^*)$ by 
\[
\big(\phi^*(u)(f)\big)(x) = -f\big(\phi(u)(x)\big) \quad \text{for all } u \in \mathcal{A}, \; f \in V^*, \; x \in V.
\]

\sssbegin{Proposition}\label{rep1}
Let $(\mathcal{A}, \bullet, \omega)$ be a pre-symplectic left-symmetric algebra. Then the pair $\big(-(\Ll^\bullet)^*, (\Rr^\star)^*\big)$ defines a representation of $\mathcal{A}$ on $\mathcal{A}^*$.
\end{Proposition}

\begin{proof}
According to Proposition \ref{condition plate} $(\A, \bullet)$ is  left-symmetric $\Ll$-algebra. Let $u,v \in \A$ and $f \in \A^*$. By the identities  (b) and (c) of Proposition~\ref{Pr associe}, we obtain
\[
(\Ll_u^\bullet)^* \circ (\Ll_v^\bullet)^*(f)
= - (\Ll_u^\bullet)^*(f \circ \Ll_v^\bullet)
=  f \circ \Ll_v^\bullet \circ \Ll_u^\bullet.
\]
Using again the left-symmetric $\Ll$-algebra identities, this yields
\[
(\Ll_u^\bullet)^* \circ (\Ll_v^\bullet)^*
=  (\Ll_v^\bullet)^* \circ (\Ll_u^\bullet)^*.
\]
Moreover,
\[
(\Ll_{[u, v]_\bullet}^\bullet)^*(f)
= -f \circ \Ll_{[u, v]_\bullet}^\bullet = 0,
\]
by the left-symmetric $\Ll$-algebra condition.

Moreover,  we have
\[
(\Rr_u^\star)^* \circ (\Rr_v^\star)^*(f)
= f \circ \Rr_v^\star \circ \Rr_u^\star \overset{(b)}{=}-f\circ \Ll_v^\bullet\circ \Rr_u^\star= -(\Rr_u^\star)^* \circ (\Ll_v^\bullet)^*(f)
\]
by identity $(b)$.

Similarly, we have
\[
(\Rr_{u\bullet v}^\star)^*(f)
= -f \circ \Rr_{u\bullet v}^\star = -f\circ \Rr_v^\star\circ \Ll_u^\bullet=-(\Ll_u^\bullet)^*\circ(\Rr_v^\star)(f)
\]
by identity $(c)$.

Therefore, $\big(-(\Ll^\bullet)^*, (\Rr^\star)^* \big)$ defines a representation of $\A$ on $\A^*$.
\end{proof}

\sssbegin{Proposition}\label{Prop7.8}
Let $(\mathcal{A},\bullet,\omega)$ be a pre-symplectic left-symmetric algebra. Let $\Ll^\bullet$ and $\Rr^\bullet$ denote the left and right multiplication operators, respectively, and let $\Rr^\star$ denote the right multiplication operator associated with the product $\star$ defined in \eqref{produce associe}. Then $\dim\A$ is even and there exists an isomorphism $\Phi : \mathcal{A} \longrightarrow \mathcal{A}^{*}$ such that
\[
\Phi(\Ll_u^\bullet(a)) = -(\Ll_u^\bullet)^{*}(\Phi(a)) \quad \text{and} \quad \Phi(\Rr^\bullet_u(a)) = (\Rr_u^\star)^{*}(\Phi(a)) \quad \text{for all } u,a \in \mathcal{A}.
\]
\end{Proposition}

\begin{proof}
Since $\omega$ is a non-degenerate skew-symmetric bilinear form, the vector space $\mathcal{A}$ is necessarily even-dimensional. Consider the linear map $\Phi : \mathcal{A} \longrightarrow \mathcal{A}^{*}$ defined by
$
\Phi(a)(b) := \omega(a, b)$, for all $a,b \in \mathcal{A}.$ Since $\omega$ is nondegenerate, $\Phi$ is an isomorphism of vector spaces. 

Let $u,a,b \in \mathcal{A}$. Using the symmetry of the left multiplication operators with respect to $\omega$, we obtain
\[
\Phi(\Ll_u^\bullet(a))(b) = \omega(u \bullet a, b) = \omega(a, u \bullet b) = \Phi(a)(\Ll_u^\bullet(b)) = -\big((\Ll_u^\bullet)^{*}(\Phi(a))\big)(b).
\]
Thus, we have
\[
\Phi(\Ll_u^\bullet(a)) = -(\Ll_u^\bullet)^{*}(\Phi(a)).
\]
Similarly, for the right multiplication operator, we have
\[
\Phi(\Rr^\bullet_u(a))(b) = \omega(a \bullet u, b) = \omega(a, u \star b) = -\omega(a, b \star u) = -\Phi(a)(\Rr^\star_u(b)) = \big((\Rr_u^\star)^{*}(\Phi(a))\big)(b),
\]
which implies that
\[
\Phi(\Rr^\bullet_u(a)) = (\Rr_u^\star)^{*}(\Phi(a)).
\]
This completes the proof.
\end{proof}

\sssbegin{Theorem}\label{Thm7.11}
Let $(\A,\bullet)$ be a left-symmetric algebra, and let $\Ll^\bullet$ and $\Rr^\bullet$ denote the left and right multiplication operators, respectively. Then there exists a nondegenerate skew-symmetric bilinear form 
$
\om : \A \times \A \longrightarrow \mathbb{K}
$
such that $(\A,\bullet, \om)$ is a pre-symplectic left-symmetric algebra if and only if the dimension of $\A$ is even and  there exists an isomorphism 
$
\Phi : \A \longrightarrow \A^{*}
$
satisfying \[
\Phi(\Ll_u^\bullet(a)) = -(\Ll_u^\bullet)^{*}(\Phi(a)) \; \text{and} \; \Phi(\Rr^\bullet_u(a)) = (\Rr_u^\star)^{*}(\Phi(a)), 
\]
where the product $\star$  is skew-symmetric. 
\end{Theorem}
\begin{proof}
Assume first that $(\A,\bullet, \om)$ is a pre-symplectic left-symmetric algebra. Then the result follows directly from Proposition~\ref{Prop7.8}.

Conversely, suppose that there exists an isomorphism $\Phi : \A \to \A^{*}$ such that
\[
\Phi(\Ll_u^\bullet(a)) = -(\Ll_u^\bullet)^{*}(\Phi(a)) \; \text{and} \; \Phi(\Rr^\bullet_u(a)) = (\Rr_u^\star)^{*}(\Phi(a)), \text{ for all $u,a \in \A$.}
\]
Consider the bilinear form
$
T : \A \times \A \longrightarrow \mathbb{K}$ defined by $T(u,v) := \Phi(u)(v),$
for all $u,v \in \A$. Since $\Phi$ is an isomorphism, $T$ is a nondegenerate bilinear form.

Moreover, for all  $u,v,w \in \A$, we have
\[
\begin{aligned}
T(u \bullet v, w) &= \Phi(u \bullet v)(w) 
= \Phi(\Ll_u^\bullet(v))(w) = -(\Ll^\bullet_u)^*(\Phi(v))(w)=\Phi(v)(\Ll^\bullet_u(w))=T(v, u\bullet w)
\end{aligned}
\]
Thus, we get that 
\begin{equation}\label{s1}
T(u \bullet v, w) =T(v, u\bullet w).
\end{equation}
Moreover, we have 
\[
\begin{aligned}
T(v \bullet u, w) &= \Phi(v \bullet u)(w) 
= \Phi(\Rr_u^\bullet(v))(w) = (\Rr^\star_u)^*(\Phi(v))(w)\\& =-\Phi(v)(\Rr^\star_u(w))=-T(v, w\star u)=T(v, u\star w)\\
&=- T(v\bullet w, u) \overset{\eqref{s1}}{=}-T(w, v\bullet u)
\end{aligned}
\]
Thus, we obtain  
\begin{equation}\label{s2}
T(u \bullet v, w) =- T(w, u\bullet v).
\end{equation}

Now, we introduce the symmetric and skew-symmetric parts of the bilinear form $T$. Define the bilinear forms $T_s, T_a : \A \times \A \to \mathbb{K}$ by
\[
T_s(u,v) := \frac{1}{2}\Big(T( u,v ) + T (v,u)\Big),\esp
T_a(u,v) := \frac{1}{2}\Big(T( u,v ) - T( v,u )\Big),
\]
for all  $u,v \in \A$. Clearly, $T = T_s + T_a$.

Since $T$ satisfies \eqref{s1}, for all  $u,v,w \in \A$, we obtain
\[
\begin{aligned}
T_a(u \bullet v, w) &= \frac{1}{2}\Big(T( u \bullet v, w ) - T(w, u \bullet v )\Big) \\
  &\overset{\eqref{s1}}{=} \frac{1}{2}\Big(T( v, u \bullet w) - T( u\bullet w ,v )\Big)\\
&= T_a(v, u \bullet w).
\end{aligned}
\]
It follows that $T_a(u \bullet v, w)=T_a(v, u \bullet w)$.

Since $T$ satisfies \eqref{s2}, for all  $u,v,w \in \A$, we obtain
\[
\begin{aligned}
T_s(u \bullet v, w) &= \frac{1}{2}\Big(T( u \bullet v, w ) + T(w, u \bullet v )\Big) \\
&\overset{\eqref{s2}}{=} \frac{1}{2}\Big(  -T(w, u \bullet v )-T( u \bullet v, w )\Big)
\end{aligned}
\]
Therefore,
\[
T_s(u \bullet v, w) = - T_s(u \bullet v, w), \text{
which implies } 
T_s(u \bullet v, w) = 0.
\]
We define the two spaces:
\[
N := \{u \in \A \mid T_s(u,v)=0 \ \text{for all } v \in \A\} \text{ and } W := \{u \in \A \mid T_a(u,v)=0 \ \text{for all } v \in \A\}.
\] 
By the previous argument, we have $\A^2=\A\bullet\A \subseteq N$, and consequently $N$ is a two-sided ideal of $(\A, \bullet)$. Since $T$ is nondegenerate, we have $N \cap W = \{0\}$. Hence, there exists a vector subspace $V \subseteq \A$ such that $
\A = W \oplus V,$ and clearly $N \subseteq V$. Moreover, since $\A^2 \subseteq N \subseteq V$, it follows that $V$ is a two-sided ideal of $(\A, \bullet)$ and that $T_a|_{V \times V}$ is nondegenerate. Consequently, $V$ is even-dimensional. Since $\mathcal{A}$ is also even-dimensional, $W$ is even-dimensional as well.

Consider a nondegenerate skew-symmetric bilinear form 
$
H : W \times W \longrightarrow \mathbb{K}.
$ We define a bilinear form $\widetilde{H} : \A \times \A \to \mathbb{K}$ by
$
\widetilde{H}|_{W \times W} = H,$ and $ \widetilde{H}(V,\A)=\widetilde{H}(\A,V)=0.
$
Clearly, $\widetilde{H}$ is skew-symmetric. Now, define the bilinear form
\[
\om := T_a + \widetilde{H} : \A \times \A \longrightarrow \mathbb{K}.
\]

Since $T_a$ and $\widetilde H$ are
skew-symmetric, the form
$\omega=T_a+\widetilde H$
is skew-symmetric. We will show that it is actually nondegenerate.

Let $u = w + v \in \A$, with $w \in W$ and $v \in V$, and assume that
$
\om( u,z)=0,$  for all $z \in \A.$
Then
$
\widetilde{H}(w,z) + T_a(v,z)=0.$

If we choose $z\in W$, then $T_a(v,z)=0$, hence $\widetilde{H}(w,z)=0$, and thus $H(w,z)=0$. Since $H$ is nondegenerate, it follows that $w=0$.

If we choose $z\in V$, then $\widetilde{H}(w,z)=0$, and therefore $T_a(v,z)=0$. Since $T_a|_{V \times V}$ is nondegenerate, it follows that $v=0$.

Thus $u=0$, and consequently $\om$ is nondegenerate.

Finally, since $\A^2 \subseteq V$, $\widetilde{H}(V,\A)=\widetilde{H}(\A,V)=0$, and $T_a$ satisfies identity \eqref{s1}, we obtain, for all $u,v,w \in \A$, 
\[
\om( u \bullet v, w)  = T_a(u \bullet v, w) = T_a(v, u \bullet w) = \om( v, u\bullet w) .
\]
Therefore, $(\A,\bullet, \om)$ is a pre-symplectic left-symmetric algebra.
\end{proof}

\sssbegin{Definition}
Let $(\A, \cdot)$ be a non-associative algebra, and let $(V_1,  l_1, r_1)$ and $(V_2, l_2, r_2)$ be two $\A$-bimodules. A  linear map $\Phi : V_1 \to V_2$ is called a \emph{morphism of $\A$-bimodules} if, for all $u \in \A$ and $x \in V_1$, the following identities hold:
\[
\Phi\big(l_1(u)(x)\big)
=  l_2(u)\big(\Phi(x)\big)\esp \Phi\big(r_1(u)(x)\big)
=   r_2(u)\big(\Phi(x)\big).
\]
Moreover, $\Phi$ is called an \emph{isomorphism of $\A$-bimodules} if it is bijective.
\end{Definition}

A consequence of the previous Theorem~\ref{Thm7.11} is the following result for $\mathcal{A}$-bimodules.

\sssbegin{Corollary}\label{Cor_bimodules}
Let $(\mathcal{A}, \bullet)$ be a left-symmetric algebra, and let $\Ll^\bullet$ and $\Rr^\bullet$ denote the operators of left and right multiplication, respectively. Then there exists a nondegenerate skew-
symmetric bilinear form $\omega: \mathcal{A} \times \mathcal{A} \to \mathbb{K}$ such that $(\mathcal{A},\bullet, \omega)$ is a pre-symplectic left-symmetric algebra if and only if the dimension of $\mathcal{A}$ is even and there exists a skew-symmetric product $\star$ on $\mathcal{A}$ such that the $\mathcal{A}$-bimodules $(\mathcal{A}, \Ll^\bullet, \Rr^\bullet)$ and $(\mathcal{A}^*, -(\Ll^\bullet)^*, (\Rr^\star)^*)$ are isomorphic.
\end{Corollary}

Motivated by the previous results, we now present a new construction of pre-symplectic  left-symmetric algebras.

\sssbegin{Theorem}\label{construction}
Let $(\mathcal{A},\bullet)$ be a left-symmetric $\mathrm{L}$-algebra. Assume that a bilinear map $\star : \mathcal{A} \times \mathcal{A} \to \mathcal{A}$ satisfies identities $(a)$, $(b)$, and $(c)$ of Proposition~\ref{Pr associe}. Then the vector space $T^*(\mathcal{A}) = \mathcal{A} \oplus \mathcal{A}^*$, endowed with the product
\[
(u+f)\diamond (v+g) = u \bullet v - (\Ll_u^\bullet)^*(g) + (\Rr_v^\star)^*(f),
\]
for all $u,v\in \mathcal{A}$ and $f,g\in \mathcal{A}^*$, is a left-symmetric $\mathrm{L}$-algebra.

Moreover, the bilinear form $\omega_c : T^*(\mathcal{A}) \times T^*(\mathcal{A}) \to \mathbb{K}$ defined by
\[
\omega_c(u+f, v+g) = f(v) - g(u)
\]
is nondegenerate and skew-symmetric, and the left multiplication operators associated with $\diamond$ are symmetric with respect to $\omega_c$. Consequently, the triple $(T^*(\mathcal{A}), \diamond, \omega_c)$ is a pre-symplectic left-symmetric algebra, called the $T^*$-extension of $(\mathcal{A}, \bullet)$.

In this case, we have
\[
\omega_c\big((u+f)\diamond (v+g), w+h\big) = \omega_c\big(u+f, (v+g)\triangleright (w+h)\big),
\]
where
\[
(v+g)\triangleright (w+h) = v\star w - g\circ \Rr_w^\bullet + h\circ\Rr_v^\bullet \quad \text{for all } v,w\in \mathcal{A} \text{ and } g,h\in \mathcal{A}^*.
\]
\end{Theorem}

\begin{proof}
By Proposition~\ref{rep1}, the pair $\big(-(\Ll^\bullet)^*, (\Rr^\star)^*\big)$ defines a representation of $\mathcal{A}$ on $\mathcal{A}^*$. Therefore, by the semidirect product construction, $(T^*(\mathcal{A}), \diamond)$ is a left-symmetric algebra.

By definition, the bilinear form $\omega_c$ is nondegenerate and skew-symmetric. Let $u,v,w\in \mathcal{A}$ and $f,g,h\in \mathcal{A}^*$. Then, we have
\begin{align*}
\om_c( (u+f)\diamond (v+g), w+h )
&= \om_c( u\bullet v - (\Ll_u^\bullet)^*(g) +(\Rr_v^\star)^*(f),\, w+h ) \\
&= g(u \bullet w) -  f(w \star v) -  h(u \bullet v) \\
&=  g(u \bullet w) -  h\circ \Ll_u^\bullet (v) +f\circ \Rr^\star_w (v)   \\
&= \om_c(v+g, u\bullet w+h\circ \Ll_u^\bullet-f\circ \Rr^\star_w)\\&= \om_c(v+g, (u+f)\diamond (w+h)).
\end{align*}
Therefore, the left multiplication operators associated with $\diamond$ are symmetric with respect to $\omega_c$, which implies that $(T^*(\mathcal{A}), \diamond, \omega_c)$ is a pre-symplectic left-symmetric algebra.

On the other hand, to determine the operation $\triangleright$, we have:
\begin{align*}
\om_c( u+f, (v+g)\rhd (w+h) )
&= \om_c( (u+f)\diamond (v+g), w+h)\\&= \om_c( u\bullet v - (\Ll_u^\bullet)^*(g) +(\Rr_v^\star)^*(f),\, w+h ) \\
&=  g(u \bullet w) -  f(w \star v) -  h(u \bullet v) 
\\&= \om_c( u+f, v\star w-g\circ \Rr_w^\bullet +h\circ\Rr_v^\bullet).
\end{align*}
Since $\om_c$ is nondegenerate, we conclude that 
$$
(v+g)\rhd (w+h)= v\star w-g\circ \Rr_w^\bullet +h\circ\Rr_v^\bullet.$$
This completes the proof.
\end{proof}

\sssbegin{Corollary}\label{PP}
Let $(\mathcal{A}, \bullet, \omega)$ be a pre-symplectic left-symmetric algebra. Then $(T^*(\mathcal{A}), \diamond, \omega_c)$ is a pre-symplectic left-symmetric algebra, where the product $\star$ is defined in \eqref{produce associe}.
\end{Corollary}

\begin{proof}
Since the product $\star$ defined in \eqref{produce associe} satisfies identities $(a)$, $(b)$, and $(c)$ of Proposition~\ref{Pr associe}, the result follows immediately from Theorem~\ref{construction}.
\end{proof}

\section{Symplectic Lie algebras}
In this section, we study the Levi-Civita product associated with symplectic Lie algebras under the condition that it defines a left-symmetric algebra. We also provide a characterization of symplectic Lie algebras in terms of representations of left-symmetric algebras.

Recall that a \emph{symplectic Lie algebra} is a pair $(\g,\omega)$, where $\g$ is a Lie algebra and $\omega$ is a nondegenerate skew-symmetric bilinear form satisfying
\[
\omega([u,v],w)+\omega([v,w],u)+\omega([w,u],v)=0,
\qquad \text{for all }\, u,v,w\in\g.
\]
This condition means that $\omega$ is a $2$-cocycle in the Lie algebra cohomology.

It is known (see \cite{ch}) that the product $\cdot$ uniquely defined by
\begin{equation}
\omega(u\cdot v,w)=-\omega(v,[u,w]),
\qquad \text{for all }\, u,v,w\in\g,
\label{Produit as1}
\end{equation}
induces a left-symmetric algebra structure on $\g$, satisfying
\[
u\cdot v-v\cdot u=[u,v].
\]
The product $\cdot$ is called the left-symmetric product associated with the symplectic Lie algebra $(\g,\br,\omega)$.

\ssbegin{Proposition}\label{symmetrique a droite}
Let $(\g, \br, \omega)$ be a symplectic Lie algebra. Then, the Levi-Civita product $\bullet$ is given by
\begin{equation}
\omega(u\bullet v,w)=\omega(u,[v,w]),
\qquad \text{for all } u,v,w\in\g.
\label{Produit as}
\end{equation}
In particular, for any $u,v\in\g$, we have
$
u\bullet v = -v\cdot u,
$ where $\cdot$ is the left-symmetric product defined by~\eqref{Produit as1}. Moreover, $(\g,\bullet)$ is a right-symmetric algebra.
\end{Proposition}

\begin{proof}
By Proposition~\ref{pdoduit de Levi-civita}, the Levi-Civita product satisfies the following identity for all $u,v,w\in\g$:
\[
2\omega(u\bullet v,w) = \omega([u,v],w) - \omega([v,w],u) + \omega([w,u],v).
\]
Since $\omega$ is a $2$-cocycle, it satisfies the identity:
\[
\omega([u,v],w) = -\omega([v,w],u) - \omega([w,u],v).
\]
Substituting this identity into the previous equation yields:
\[
\omega(u\bullet v,w) = \omega(u,[v,w]) = -\omega(v\cdot u,w).
\]
Since $\omega$ is nondegenerate, it follows that $u\bullet v = -v\cdot u$ for all $u,v\in\g$.

Finally, using the fact that $(\g,\cdot)$ is a left-symmetric algebra along with the relation $u\bullet v = -v\cdot u$, it follows immediately that $(\g,\bullet)$ is a right-symmetric algebra.
\end{proof}

\ssbegin{Proposition}
Let $(\g,\br,\omega)$ be a symplectic Lie algebra of dimension $2n$, and let $\bullet$ be its Levi-Civita product. If $(\g,\bullet)$ is a left-symmetric algebra, then
$
\dim [\g,\g]\leq n,
$
and the derived ideal $[\g,\g]$ is totally isotropic.
\end{Proposition}
\begin{proof}
By Prop. \ref{condition plate}, we have
$[u,v]\bullet w=0,$
for all $u,v,w\in\g.$ Using \eqref{Produit as}, we obtain
\[
0=\omega([u,v]\bullet w,z)=\omega([u,v],[w,z]),
\qquad \forall\, u,v,w,z\in\g,\]
which implies that
$[\g,\g]\subseteq [\g,\g]^{\perp}.$ Hence, the derived ideal $[\g,\g]$ is totally isotropic. Since $\omega$ is nondegenerate and $\dim\g=2n$, it follows that
$
\dim[\g,\g]\leq n.
$
\end{proof}

\ssbegin{Theorem}
Let $(\g, \br, \om)$ be a symplectic Lie algebra, let $\bullet $ be its Levi-Civita product, and let $\cdot$ be the left-symmetric product associated with the symplectic Lie algebra $(\g, \br, \om)$. Then, the following assertions are equivalent:
\begin{enumerate}
\item $(\g, \bullet)$ is a left-symmetric algebra.
\item $(\g, \bullet)$ is a left-symmetric $\Ll$-algebra.
\item $(\g, \cdot)$ is a Novikov algebra.
\item $(\g, \cdot)$ and $(\g, \bullet)$ are associative algebras.
\end{enumerate}
\end{Theorem}

\begin{proof}
According to Proposition~\ref{condition plate}, assertions (1) and (2) are equivalent.

Assume that (2) holds. By Proposition~\ref{symmetrique a droite}, the product $\bullet$ is right-symmetric. Moreover, the product defined by
$
u\cdot v=-v\bullet u$, for all $u,v\in\g,$ is left-symmetric. Since $(\g,\bullet)$ is an $\Ll$-algebra, it follows that $(\g,\cdot)$ is a Novikov algebra. Thus, (3) holds.

By \cite[Theorem~1]{Wadia}, the product $\cdot$ is associative. Since
$
u\bullet v=-v\cdot u$, for all $u,v\in\g,$ the product $\bullet$ is also associative. Hence, (4) holds.

If (4) holds, then $(\g,\bullet)$ is associative and therefore left-symmetric, proving (1). This completes the proof.
\end{proof}

\ssbegin{Proposition}
Let $(\g, \br, \omega)$ be a symplectic Lie algebra. If its Levi-Civita product $\bullet$ is a Novikov algebra, then $(\g, \br)$ is an at most $2$-step nilpotent Lie algebra.
\end{Proposition}

\begin{proof}
Assume that $(\g,\bullet)$ is a Novikov algebra. According to Proposition~\ref{Condition Novikov}, we have:
\[
(u\bullet v)\bullet w = 0,
\qquad \text{for all } u,v,w\in\g.
\]
Using equation~\eqref{Produit as}, we obtain for all $u,v,w,z\in\g$:
\[
0 = \omega\big((u\bullet v)\bullet w, z\big)
= \omega(u\bullet v, [w,z])
= \omega(u, [v, [w,z]]).
\]
Since $\omega$ is nondegenerate, it follows that $[v, [w,z]] = 0$ for all $v,w,z\in\g$. Hence, $[\g, [\g, \g]] = 0$, which proves that $(\g, \br)$ is a $2$-step nilpotent Lie algebra.
\end{proof}

Recall that if $(\mathcal{A}, \bullet)$ is a left-symmetric algebra, $V$ is a vector space, and $l, r : \mathcal{A} \longrightarrow \mathrm{End}(V)$ are two linear maps, then $(l, r)$ is a representation of $\mathcal{A}$ in $V$ if and only if, for all $u,v \in \mathcal{A}$, the following identities hold: 
\begin{equation}\label{eq3.18}
l([u, v]_\bullet) = [l(u),l(v)], \qquad r(u\bullet v) - r(v)\circ r(u) = [l(u), r(v)].
\end{equation}

\ssbegin{Proposition}\label{rep2}
Let $(\g, \br, \omega)$ be a symplectic Lie algebra, and let $\cdot$ denote the left-symmetric product associated with $(\g, \br, \omega)$. Then, the pair $\big(\ad_u^*, -(\Rr_u^\cdot)^*\big)$ defines a representation of the left-symmetric algebra $(\g,\cdot)$ on $\g^{*}$.
\end{Proposition}

\begin{proof}
The proof is analogous to that of Proposition~\ref{rep1}.
\end{proof}

\ssbegin{Proposition}\label{Pr19}
Let $(\g, \br, \omega)$ be a symplectic Lie algebra, and let $\cdot$ denote the left-symmetric product associated with $(\g, \br, \omega)$. Let $\Ll^\cdot_u$ and $\Rr^\cdot_u$ denote the left and right multiplication operators, respectively. Then, $\dim\A$ is even and there exists an isomorphism $\Phi : \g \longrightarrow \g^{*}$ such that
\[
\Phi(\Rr_u^\cdot(a)) = -(\Rr_u^\cdot)^{*}(\Phi(a)) \quad \text{and} \quad \Phi(\Ll^\cdot_u(a)) = \ad_u^*(\Phi(a)),
\]
for all $u,a \in \g$.
\end{Proposition}

\begin{proof}
The proof is analogous to that of Proposition~\ref{Prop7.8}.
\end{proof}

\ssbegin{Theorem}\label{Thm20}
Let $(\mathcal{A},\cdot)$ be a left-symmetric algebra, and let $\Ll^\cdot$ and $\Rr^\cdot$ denote the left and right multiplication operators, respectively. Let $\ad_u$ denote the adjoint operator of the Lie algebra associated with $(\mathcal{A},\cdot)$. Then, there exists a nondegenerate skew-symmetric bilinear form
$\omega : \mathcal{A} \times \mathcal{A} \longrightarrow \mathbb{K}$
satisfying \eqref{Produit as1} if and only if $\dim\mathcal{A}$ is even and there exists an isomorphism 
$\Phi : \mathcal{A} \longrightarrow \mathcal{A}^{*}$
satisfying:
\[
\Phi(\Rr_u^\cdot(a)) = -(\Rr_u^\cdot)^{*}(\Phi(a)) \quad \text{and} \quad \Phi(\Ll^\cdot_u(a)) = \ad_u^{*}(\Phi(a)).
\]

In this case, $(\mathcal{A}^-, \omega)$ is a symplectic Lie algebra.
\end{Theorem}

\begin{proof}
Assume first that there exists a non-degenerate skew-symmetric bilinear form $\omega$ satisfying \eqref{Produit as1}. Then the result follows directly from Proposition~\ref{Pr19}.

Conversely, suppose there exists an isomorphism $\Phi : \mathcal{A} \to \mathcal{A}^{*}$ satisfying the given relations for all $u,a \in \mathcal{A}$. Consider the bilinear form $T : \mathcal{A} \times \mathcal{A} \longrightarrow \mathbb{K}$ defined by $T(u,v) := \Phi(u)(v)$ for all $u,v \in \mathcal{A}$. Since $\Phi$ is an isomorphism, $T$ is nondegenerate.

Moreover, for all $u,v,w \in \mathcal{A}$, we have:
\[
\begin{aligned}
T(v\cdot u, w) &= \Phi(v \cdot u)(w) = \Phi(\Rr_u^\cdot(v))(w) \\
&= -(\Rr^\cdot_u)^*(\Phi(v))(w) = \Phi(v)(\Rr^\cdot_u(w)) = T(v, w\cdot u).
\end{aligned}
\]
Thus, we obtain:
\begin{equation}\label{N1}
T(v \cdot u, w) = T(v, w\cdot u).
\end{equation}
Similarly, we have:
\[
\begin{aligned}
T(u \cdot v, w) &= \Phi(u \cdot v)(w) = \Phi(\Ll_u^\cdot(v))(w) = \ad_u^*(\Phi(v))(w) \\
&= -\Phi(v)(\ad_u(w)) = -T(v, [u,w]) \\
&= -T(w\cdot v, u) \overset{\eqref{N1}}{=} -T(w, u \cdot v).
\end{aligned}
\]
Thus, we get:
\begin{equation}\label{N2}
T(u\cdot v,w) = -T(v,[u,w]) = -T(w,u\cdot v).
\end{equation}

Now, we decompose the bilinear form $T$ into its symmetric and skew-symmetric parts. Define the bilinear forms $T_s, T_a : \mathcal{A} \times \mathcal{A} \to \mathbb{K}$ by:
\[
T_s(u,v) := \frac{1}{2}\Big(T(u,v) + T(v,u)\Big), \qquad
T_a(u,v) := \frac{1}{2}\Big(T(u,v) - T(v,u)\Big),
\]
for all $u,v \in \mathcal{A}$, so that $T = T_s + T_a$.

Since $T$ satisfies \eqref{N2}, we have:
\[
T_s(u\cdot v,w) = \frac{1}{2}\Big(T(u\cdot v,w) + T(w,u\cdot v)\Big) \overset{\eqref{N2}}{=} \frac{1}{2}\Big(-T(w,u\cdot v) - T(u\cdot v,w)\Big).
\]
Therefore, $T_s(u\cdot v,w) = -T_s(u\cdot v,w)$, which implies:
\begin{equation}
T_s(u\cdot v,w) =T([u,v], w)= 0.
\label{Ts}\end{equation}
Consequently, using $T = T_s + T_a$ together with \eqref{Ts} and \eqref{N2}, we obtain:
\begin{equation}\label{Ta}
T_a(u \cdot v, w) = T(u \cdot v, w) \overset{\eqref{N2}}{=} -T(v, [u,w]) = -T_a(v, [u,w]).
\end{equation}

We define the two spaces:
\[
N := \{u \in \A \mid T_s(u,v)=0 \ \text{for all } v \in \A\} \text{ and } W := \{u \in \A \mid T_a(u,v)=0 \ \text{for all } v \in \A\}.
\] 
By the previous argument, we have $\A^2=\A\cdot\A \subseteq N$, and consequently $N$ is a two-sided ideal of $(\A, \cdot)$. Since $T$ is nondegenerate, we have $N \cap W = \{0\}$. Hence, there exists a vector subspace $V \subseteq \A$ such that $
\A = W \oplus V,$ and clearly $N \subseteq V$. Moreover, since $\A^2 \subseteq N \subseteq V$, it follows that $V$ is a two-sided ideal of $(\A, \cdot)$ and that $T_a|_{V \times V}$ is nondegenerate. Consequently, $V$ is even-dimensional. Since $\mathcal{A}$ is also even-dimensional, $W$ is even-dimensional as well.

Consider a nondegenerate skew-symmetric bilinear form 
$
H : W \times W \longrightarrow \mathbb{K}.
$ We define a bilinear form $\widetilde{H} : \A \times \A \to \mathbb{K}$ by
$
\widetilde{H}|_{W \times W} = H,$ and $ \widetilde{H}(V,\A)=\widetilde{H}(\A,V)=0.
$
Clearly, $\widetilde{H}$ is skew-symmetric. Now, define the bilinear form
\[
\om := T_a + \widetilde{H} : \A \times \A \longrightarrow \mathbb{K}.
\]
By construction, $\omega$ is a non-degenerate skew-symmetric bilinear form on $\mathcal{A}$.

Finally, since $\mathcal{A}^2 \subseteq V$ and $\widetilde{H}(\mathcal{A}^2, \mathcal{A}) = \widetilde{H}(\mathcal{A}, \mathcal{A}^2) = 0$, we apply \eqref{Ta} to deduce that, for all $u, v, w \in \mathcal{A}$:
\[
\omega(u \cdot v, w) = T_a(u \cdot v, w) \overset{\eqref{Ta}}{=} -T_a(v, [u, w]) = -\omega(v, [u, w]).
\]
This completes the proof.

\end{proof}

A consequence of Theorem~\ref{Thm20} is the following result for $\mathcal{A}$-bimodules.

\ssbegin{Corollary}
Let $(\mathcal{A},\cdot)$ be a left-symmetric algebra, and let $\Ll^\cdot$ and $\Rr^\cdot$ denote the left and right multiplication operators, respectively. Let $\ad_u$ denote the adjoint operator of the Lie algebra associated with $(\mathcal{A},\cdot)$. Then, there exists a nondegenerate skew-symmetric bilinear form
$\omega : \mathcal{A} \times \mathcal{A} \longrightarrow \mathbb{K}$
satisfying \ref{Produit as1} if and only if $\dim\mathcal{A}$ is even and the left-symmetric $\mathcal{A}$-bimodules $(\mathcal{A},\Ll^\cdot,\Rr^\cdot)$ and $(\mathcal{A}^{*}, \ad^*, -(\Rr^\cdot)^{*})$ are isomorphic.
\end{Corollary}

\ssbegin{Theorem}\label{constriction12} 
Let $(\mathcal{A},\cdot)$ be a left-symmetric algebra, and let $\Ll^\cdot$ denote the left multiplication operator. Then, the vector space $T^*(\mathcal{A}) = \mathcal{A} \oplus \mathcal{A}^*$, endowed with the bracket
\begin{equation}\label{T-ex}
[u+f, \, v+g] = [u,v]_{\cdot} + \Ll_u^*(g) - \Ll_v^*(f),
\end{equation}
for all $u,v \in \mathcal{A}$ and $f,g \in \mathcal{A}^*$, is a Lie algebra, 

Moreover, the bilinear form $\omega_c : T^{*}(\mathcal{A}) \times T^{*}(\mathcal{A}) \to \mathbb{K}$ defined by
\[
\omega_c(u+f, \; v+g) = f(v) - g(u),
\]
for all $u,v\in\mathcal{A}$ and $f,g\in\mathcal{A}^{*}$, is skew-symmetric, nondegenerate, and a $2$-cocycle on the Lie algebra $(T^{*}(\mathcal{A}), \br)$. Consequently, $(T^{*}(\mathcal{A}), \br, \omega_c)$ is a symplectic Lie algebra.

In this case, the symplectic Lie algebra $(T^*(\A), \br , \omega_c)$ is called the $T^*$-extension of $\A$.
\end{Theorem}

\begin{proof}
The assertion follows from the cotangent semidirect-product
construction in \cite[Section~3.2]{Wadia}, with the conventions for the order of the summands and the sign of the symplectic form used here.
\end{proof}

\ssbegin{Corollary}\label{PP}
Let $(\g, \br, \omega)$ be a symplectic Lie algebra, and let $\cdot$ denote the associated left-symmetric product defined by~\eqref{Produit as1}. Then, $(T^{*}(\g), \br, \omega_c)$ is a symplectic Lie algebra.
\end{Corollary}

\begin{proof}
Since the product $\cdot$ defined by~\eqref{Produit as1} endows $\g$ with a left-symmetric algebra structure, the assertion follows directly from Theorem~\ref{constriction12}.
\end{proof}

\section{Some extensions of left-symmetric algebras}
In this section, we study the double extension process, which consists of two successive and interdependent extensions. The first is an annihilator extension of a left-symmetric $\Ll$-algebra, while the second is an almost-semidirect product of left-symmetric $\Ll$-algebras.

\subsection{Annihilator extensions of left-symmetric $\mathrm{L}$-algebras}

Let $(\mathcal{A}, \bullet)$ be a left-symmetric $\mathrm{L}$-algebra, let $\mathfrak{h}$ be a vector space, and let $\mu : \mathcal{A} \times \mathcal{A} \rightarrow \mathfrak{h}$ be a bilinear map. We define a new product $\diamond$ on the vector space $\widetilde{\mathcal{A}} := \mathcal{A} \oplus \mathfrak{h}$ as follows:
\[
(u+a)\diamond(v+b) := u \bullet v + \mu(u,v) \qquad \text{for all } u,v \in \mathcal{A}, \; a,b \in \mathfrak{h}.
\]
Then, $(\widetilde{\mathcal{A}}, \diamond)$ is a left-symmetric $\mathrm{L}$-algebra if and only if the following conditions hold ( for all $u,v,w\in \mathcal{A}$):
\begin{equation}\label{con-left}
\mu([u,v]_\bullet, w) = 0 \quad \text{and} \quad \mu(u,v\bullet w) = \mu(v, u\bullet w).
\end{equation}

In this case, $(\widetilde{\mathcal{A}}, \diamond)$ is called the \emph{annihilator extension} of $(\mathcal{A}, \bullet)$ by means of $\mu$.

\subsection{Almost-semidirect products of left-symmetric $\mathrm{L}$-algebras}

Let $(\mathcal{A}, \bullet)$ be a left-symmetric $\mathrm{L}$-algebra and let $(\mathfrak{h}, \circ)$ be a commutative associative algebra. Consider two linear maps $F , G: \mathfrak{h} \rightarrow \mathrm{End}(\mathcal{A})$ and a bilinear map $\theta : \mathfrak{h}\times\mathfrak{h} \rightarrow \mathcal{A}$. On the vector space $\overline{\mathcal{A}} := \mathfrak{h} \oplus \mathcal{A}$, we define a new product $\bar{\bullet}$ as follows:
\[
a \bar{\bullet} b := a\circ b + \theta(a,b), \qquad a \bar{\bullet} u := F(a)u, \qquad u \bar{\bullet} a := G(a)u, \qquad u \bar{\bullet} v := u \bullet v,
\]
for all $a,b\in \mathfrak{h}$ and $u,v \in \mathcal{A}$.  

Then $(\overline{\mathcal{A}}, \bar{\bullet})$ is a left-symmetric $\mathrm{L}$-algebra if and only if the following conditions are satisfied for all $u,v \in \mathcal{A}$ and $a,b,c \in \mathfrak{h}$:
\begin{equation}\label{produit-semi}
\begin{aligned}
G(a)([u,v]_\bullet) &= 0, & u\bullet G(a)v &= v\bullet G(a)u, \\
G(a)(u)\bullet v &= F(a)(u)\bullet v, & F(a)(u\bullet v ) &= u\bullet F(a)(v), \\
G(b)G(a) &= G(b)F(a), & F(a)F(b) &= F(b)F(a), \\
F(a)G(b)u &= G(a\circ b)u + u\bullet \theta(a,b), & \theta(a, b)\bullet u &= \theta(b,a)\bullet u, \\
G(c)\big(\theta(a,b)\big) &= G(c)\big(\theta(b,a)\big), & \theta(a, b\circ c) - \theta(b, a\circ c) &= F(b)\theta(a,c) - F(a)\theta(b,c).
\end{aligned}
\end{equation}

\ssbegin{Definition}
If $(\mathfrak{h}, \mathcal{A}, F, G, \theta)$ satisfies the compatibility conditions \eqref{produit-semi}, then the left-symmetric $\mathrm{L}$-algebra $(\overline{\mathcal{A}}, \bar{\bullet})$ is called the \emph{almost-semidirect product} of the left-symmetric $\mathrm{L}$-algebra $(\mathcal{A},\bullet)$ by the commutative associative algebra $(\mathfrak{h}, \circ)$ by means of $(F,G, \theta)$. In this case, the $5$-tuple $(\mathfrak{h}, \mathcal{A}, F, G, \theta)$ is referred to as an \emph{almost-semidirect product context of left-symmetric $\mathrm{L}$-algebras}.
\end{Definition}

\subsection{Double extensions of pre-symplectic left-symmetric algebras}

Let $(\mathcal{A}, \bullet_\mathcal{A}, \omega_\mathcal{A})$ be a pre-symplectic left-symmetric algebra, let $(\mathfrak{h},\circ)$ be a commutative associative algebra, and let $\mathfrak{h}^*$ denote its dual. 

The process of constructing the double extension is summarized in the following diagram:
\begin{equation*}
\begin{tikzcd}
(\mathcal{A}, \bullet_\mathcal{A}, \omega_\mathcal{A}) \arrow[r, "\text{Step 1}"] & 
(\mathcal{A}, \bullet) \arrow[r, "\text{Step 2}"] & 
(\widetilde{\mathcal{A}} := \mathcal{A} \oplus \mathfrak{h}^*, \widetilde{\bullet}) \arrow[r, "\text{Step 3}"] & 
(\overline{\mathcal{A}} := \mathfrak{h} \oplus \mathcal{A} \oplus \mathfrak{h}^*, \bullet) \arrow[d, "\text{Step 4}"] \\
&&& (\overline{\mathcal{A}}, \bullet, \omega)
\end{tikzcd}
\end{equation*}
In what follows, we describe this construction in detail.

\medskip

\noindent
\textbf{Step 1.} According to Proposition~\ref{condition plate}, $(\mathcal{A}, \bullet)$ is a left-symmetric $\mathrm{L}$-algebra.

\medskip

\noindent
\textbf{Step 2.} Construct the annihilator extension of the left-symmetric $\mathrm{L}$-algebra $(\mathcal{A}, \bullet)$ by the vector space $\mathfrak{h}^*$.

\medskip

\noindent
\textbf{Step 3.} We then construct the almost-semidirect product of $\mathcal{A} \oplus \mathfrak{h}^*$ by the commutative associative algebra $(\mathfrak{h}, \circ)$, thereby ensuring that $(\overline{\mathcal{A}}, \bullet)$ is a left-symmetric $\mathrm{L}$-algebra.

\medskip

\noindent
\textbf{Step 4.} We extend the bilinear form $\omega_{\mathcal{A}}$ to $\omega$ on $\overline{\mathcal{A}}$ and verify that the left multiplications associated with $\bullet$ are symmetric with respect to $\omega$. Consequently, we obtain a pre-symplectic left-symmetric algebra $(\overline{\mathcal{A}}, \bullet, \omega)$.

\medskip

In order to define the double extension of $(\mathcal{A}, \bullet_\mathcal{A}, \omega_\mathcal{A})$ by $(\mathfrak{h}, \circ)$, we consider linear maps $F, G : \mathfrak{h} \to \mathrm{End}(\mathcal{A})$, a bilinear map $\theta : \mathfrak{h} \times \mathfrak{h} \to \mathcal{A}$, a skew-symmetric bilinear map $\Omega : \mathfrak{h} \times \mathfrak{h} \to \mathcal{A}$, and a bilinear map $L: \mathfrak{h} \times \mathfrak{h} \to \mathfrak{h}^*$ such that $F(a)$ is symmetric with respect to $\omega_{\mathcal{A}}$ for all $a\in \mathfrak{h}$, and $L(a, b)(c) = -L(a, c)(b)$ for all $a,b,c\in\mathfrak{h}$.

To complete the construction, we introduce the following additional linear and bilinear maps:
\[
\begin{cases}
\mu : \mathcal{A} \times \mathcal{A} \longrightarrow \mathfrak{h}^*, \qquad R : \mathcal{A} \times \mathfrak{h} \longrightarrow \mathfrak{h}^*, \\
S : \mathfrak{h} \times \mathcal{A} \longrightarrow \mathfrak{h}^*, \qquad N : \mathfrak{h} \longrightarrow \mathrm{End}(\mathfrak{h}^*), \\
\widetilde{F}, \widetilde{G} : \mathfrak{h} \longrightarrow \mathrm{End}(\mathcal{A} \oplus \mathfrak{h}^*), \quad \widetilde{\theta} : \mathfrak{h} \times \mathfrak{h} \longrightarrow \mathcal{A} \oplus \mathfrak{h}^*,
\end{cases}
\]
which are defined for all $u, v \in \mathcal{A}$, $a, b, c \in \mathfrak{h}$, and $f \in \mathfrak{h}^*$ by:
\begin{equation}\label{eq:def-maps}
\begin{cases}
N(a)(f)(b) := f(a \circ b), \\[3pt]
\mu(u, v)(a) := -\omega_\mathcal{A}(G(a)u, v), \\[3pt]
S(a, u)(b) := \omega_\mathcal{A}(u, \theta(a,b)), \\[3pt]
R(u, a)(b) := \omega_\mathcal{A}(\Omega(a,b), u), \\[3pt]
\widetilde{G}(a)(u + f) := G(a)u + R(u, a), \\[3pt]
\widetilde{F}(a)(u + f) := F(a)u + S(a, u) + N(a)(f), \\[3pt]
\widetilde{\theta}(a, b) := \theta(a, b) + L(a, b).
\end{cases}
\end{equation}

\sssbegin{Lemma}\label{le1} \label{pr-central1}
Let $\Tilde{\A} := \A \oplus \h^*$ be the vector space equipped with the product
\[
(u + f) \widetilde{\bullet} (v + g) := u \bullet_\A v + \mu(u, v), 
\quad \forall u, v \in \A, \; f, g \in \h^*.
\]
Then $(\widetilde{\mathcal{A}}, \widetilde{\bullet})$ is a left-symmetric $\mathrm{L}$-algebra if and only if the following identities hold for all $u, v \in \mathcal{A}$ and $a \in \mathfrak{h}$:
\begin{equation}
G(a)([u,v]_{\bullet_\A})=0\esp u\bullet_\A G(a)v= v\bullet_\A G(a)u
\quad \text{ for all } u, v \in \A, \; a \in \h.
\label{eq:xi-condition0}
\end{equation}
\end{Lemma}

\begin{proof}
$(\Tilde{\A}, \bullet)$ is a  left-symmetric $\Ll$-algebra if and only if the bilinear map 
\[
\mu(u,v)(a) := -\om_\A(G(a)u, v)
\]
satisfies Eq. \eqref{con-left}, which is equivalent to Eq.  \eqref{eq:xi-condition0}. 
\end{proof}

Now, let us also assume that $(\mathfrak{h}, \mathcal{A}, F, G, \theta)$ is an almost-semidirect product context of left-symmetric $\mathrm{L}$-algebras.

\sssbegin{Lemma}\label{le2}
The $5$-tuple $(\mathfrak{h}, \widetilde{\mathcal{A}}=\mathcal{A}\oplus \mathfrak{h}^*, \widetilde{F}, \widetilde{G}, \widetilde{\theta})$ defines an almost-semidirect product context
of left-symmetric L-algebras if and only if
the following identities hold,  for all $u,v\in \mathcal{A}$ and $a, b,c,l\in \mathfrak{h}$:
\begin{equation}\label{les equation}
    \begin{cases}
      G(a)^*G(b)=-G(b)^*G(a) , G(a)^*\Om(b,c)=F(a)\Om(b,c),\\ G(c)\theta(a,b)=G(c)\theta(b,a),\\ G(b)^*\theta(a,c)+\Om(a\circ c, b)
    =-G(c)^*\theta(a,b)-\Om(a\circ b, c),\\ L(a, b\circ c)(l)-L(b, a\circ c)(l)=-L(a, b\circ l)(c)+L(b, a\circ l)(c)\\ \qquad \qquad\quad\qquad \qquad\quad\qquad\quad+\om_\A(\theta(a, c), \theta(b, l))-\om_\A(\theta(b, c), \theta(a, l)),\\ \om_\A(\Om(a,b), [u,v]_{\bullet_\A})=0,\;  \om_\A(\Om(c, l), \theta(a,b)-\theta(b,a))=0.
    \end{cases}
\end{equation}
\end{Lemma}

\begin{proof}
Let $a, b, c, l \in \mathfrak{h}$, $u, v, w \in \mathcal{A}$, and $f, g \in \mathfrak{h}^*$. By using the definitions of the extended maps given in \eqref{eq:def-maps} and expansion under the compatibility conditions, straightforward computations yield the following equivalences:
\begin{itemize}
    \item The condition $\widetilde{G}(a)\big([u+f, v+g]_{\widetilde{\bullet}}\big) = 0$ holds if and only if
    \[
    G(a)([u,v]_{\bullet_\mathcal{A}}) = 0 \quad \text{and} \quad \omega_\mathcal{A}\big(\Omega(a,b), [u,v]_{\bullet_\mathcal{A}}\big) = 0.
    \]

    \item The equation $(u+f) \widetilde{\bullet} \widetilde{G}(a)(v+g) = (v+g) \widetilde{\bullet} \widetilde{G}(a)(u+f)$ is equivalent to
   \[
    u\bullet_\A G(a)(v)=v\bullet_\A G(a)(u), 
    \quad \text{and} \quad G(a)^*G(b)=-G(b)^*G(a).
    \]

    \item The relation $\widetilde{G}(a)(u+f) \widetilde{\bullet} (v+g) = \widetilde{F}(a)(u+f) \widetilde{\bullet} (v+g)$ is equivalent to
    \[
    G(a)(u) \bullet_\mathcal{A} v = F(a)(u) \bullet_\mathcal{A} v \quad \text{and} \quad G(b)G(a) = G(b)F(a).
    \]

    \item The condition $\widetilde{F}(a)\big((u+f) \widetilde{\bullet} (v+g)\big) = (u+f) \widetilde{\bullet} \widetilde{F}(a)(v+g)$ is equivalent to
    \[
    F(a)(u \bullet_\mathcal{A} v) = u \bullet_\mathcal{A} F(a)v \quad \text{and} \quad F(a)G(b)u = G(a\circ b)u + u \bullet_\mathcal{A} \theta(a,b).
    \]

    \item The identity $\widetilde{G}(b) \circ \widetilde{G}(a) = \widetilde{G}(b) \circ \widetilde{F}(a)$ is equivalent to
    \[
    G(b)G(a) = G(b)F(a) \quad \text{and} \quad G(a)^*\Omega(b,c) = F(a)\Omega(b,c).
    \]

    \item The equation $\widetilde{F}(a) \circ \widetilde{F}(b) = \widetilde{F}(b) \circ \widetilde{F}(a)$ reduces to
    \[
    F(a)F(b) = F(b)F(a) \quad \text{and} \quad F(b)\theta(a,c) + \theta(b, a\circ c) = F(a)\theta(b,c) + \theta(a, b\circ c).
    \]

    \item The relation $\widetilde{F}(a) \circ \widetilde{G}(b)(u) = \widetilde{G}(a\circ b)u + u \widetilde{\bullet} \widetilde{\theta}(a,b)$ is equivalent to
    \[
    F(a)G(b)u = G(a\circ b)u + u \bullet_\mathcal{A} \theta(a,b) \quad \text{and} \quad  -G(b)^*\theta(a,c)+\Om(b,a\circ c)
    =G(c)^*\theta(a,b)+\Om(a\circ b, c).
    \]

    \item The equation $\widetilde{\theta}(a, b) \widetilde{\bullet} (u+f) = \widetilde{\theta}(b, a) \widetilde{\bullet} (u+f)$ is equivalent to
    \[
    \theta(a,b) \bullet_\mathcal{A} u = \theta(b,a) \bullet_\mathcal{A} u \quad \text{and} \quad G(c)\theta(a,b) = G(c)\theta(b,a).
    \]

    \item The relation $\widetilde{G}(c)\widetilde{\theta}(a, b) = \widetilde{G}(c)\widetilde{\theta}(b, a)$ is equivalent to
    \[
    G(c)\theta(a,b) = G(c)\theta(b,a) \quad \text{and} \quad \omega_\mathcal{A}\big(\Omega(c, l), \theta(b,a) - \theta(a, b)\big) = 0.
    \]

    \item Finally, the  Condition $\widetilde{\theta}(a, b\circ c) - \widetilde{\theta}(b, a\circ c) = \widetilde{F}(b)\widetilde{\theta}(a,c) - \widetilde{F}(a)\widetilde{\theta}(b,c)$ is satisfied if and only if
    \[
    \theta(a, b\circ c) - \theta(b, a\circ c) = F(b)\theta(a,c) - F(a)\theta(b,c),
    \]
    and
    \[
    L(a, b\circ c)(l) - L(b, a\circ c)(l) = -L(a, b\circ l)(c) + L(b, a\circ l)(c)+\om_\A(\theta(a, c), \theta(b, l))-\om_\A(\theta(b, c), \theta(a, l)).
    \]
\end{itemize}

\end{proof}

Now, we are in a good position to introduce the notion of a \textit{double extension} of a pre-symplectic left-symmetric algebra by a commutative associative algebra.

\sssbegin{Theorem}\label{Double-extension}
Let $(\mathcal{A}, \bullet_{\mathcal{A}}, \omega_{\mathcal{A}})$ be a pre-symplectic left-symmetric algebra, and let $(\mathfrak{h}, \circ)$ be a commutative associative algebra. Assume that $(\mathfrak{h}, \mathcal{A}, F, G, \theta)$ is an almost-semidirect product context of left-symmetric $\mathrm{L}$-algebras such that, for every $a \in \mathfrak{h}$, the linear map $F(a)$ is symmetric with respect to $\omega_{\mathcal{A}}$.  Let $\Omega:\mathfrak h\times\mathfrak h
\to\mathcal A$ be a skew-symmetric
bilinear map, and let
$L:\mathfrak h\times\mathfrak h
\to\mathfrak h^*$ be a bilinear map
satisfying
\[
L(a,b)(c) = -L(a,c)(b) \quad \text{for all } a,b,c \in \mathfrak{h}.
\]
Moreover, assume that $(F,G,\theta,L,\Omega)$ satisfies the system \eqref{les equation}, and that the following identifications hold for all $a,b \in \mathfrak{h}$ and $u,v \in \mathcal{A}$:
\[
(\omega_{\mathcal{A}}(G(\cdot)u, v))(a) = \omega_{\mathcal{A}}(G(a)u, v), \; (\omega_{\mathcal{A}}(\Omega(a, \cdot), u))(b) = \omega_{\mathcal{A}}(\Omega(a,b), u)\; (\omega_{\mathcal{A}}(\theta(a, \cdot), u))(b) = \omega_{\mathcal{A}}(\theta(a,b), u).
\]
We define a bilinear product $\bullet$ on the vector space $\overline{\mathcal{A}} := \mathfrak{h} \oplus \mathcal{A} \oplus \mathfrak{h}^*$ by:
\begin{equation}\label{eq:prod-tildeA}
\begin{cases}
u \bullet v = u \bullet_{\mathcal{A}} v - \omega_{\mathcal{A}}(G(\cdot)u, v), \\[3pt]
u \bullet a = G(a)u + \omega_{\mathcal{A}}(\Omega(a, \cdot), u), \\[3pt]
a \bullet u = F(a)u - \omega_{\mathcal{A}}(\theta(a, \cdot), u), \\[3pt]
a \bullet b = a \circ b + \theta(a,b) + L(a,b), \\[3pt]
a \bullet f = f \circ \Ll_a^\circ,
\end{cases}
\end{equation}
for all $a,b \in \mathfrak{h}$, $u,v \in \mathcal{A}$, and $f \in \mathfrak{h}^*$.

Then $(\overline{\mathcal{A}}, \bullet)$ is a left-symmetric $\mathrm{L}$-algebra.

Moreover, the skew-symmetric bilinear form $\omega : \overline{\mathcal{A}} \times \overline{\mathcal{A}} \longrightarrow \mathbb{K}$ defined by
\[
\omega(u + a + f, \, v + b + g) := \omega_{\mathcal{A}}(u, v) + f(b) - g(a)
\]
is nondegenerate, and the left multiplication operators associated with $\bullet$ are symmetric with respect to $\omega$.  In this case, $(\overline{\A}, \bullet, \om)$ is a pre-symplectic left-symmetric algebra.
\end{Theorem}

The pre-symplectic left-symmetric $\mathrm{L}$-algebra $(\overline{\mathcal{A}}, \bullet, \omega)$ is called the \emph{generalized double extension} of the pre-symplectic left-symmetric algebra $(\mathcal{A}, \bullet_\mathcal{A}, \omega_{\mathcal{A}})$ by the commutative associative algebra $(\mathfrak{h}, \circ)$ by means of $(F, G, \theta, L, \Omega)$.

 We will also say that $(\mathfrak{h}, \mathcal{A}, F, G, \theta, L, \Omega)$ is a \emph{double extension context} of the pre-symplectic left-symmetric algebra $(\mathcal{A}, \bullet_\mathcal{A}, \omega_\mathcal{A})$ by the commutative associative algebra $(\mathfrak{h}, \circ)$.

\begin{proof}
Since $(\mathfrak{h}, \mathcal{A}, F, G, \theta)$ is an almost-semidirect product context of left-symmetric $\mathrm{L}$-algebras, the conditions in \eqref{eq:xi-condition0} are satisfied. Hence, the product on $\widetilde{\mathcal{A}} := \mathcal{A} \oplus \mathfrak{h}^{*}$ given for all $u,v \in \mathcal{A}$ and $f,g \in \mathfrak{h}^{*}$ by
\[
(u+f) \widetilde{\bullet} (v+g) = u \bullet_{\mathcal{A}} v - \omega_{\mathcal{A}}\big(G(\cdot)u, v\big)
\]
defines a left-symmetric $\mathrm{L}$-algebra. In other words, this constitutes an annihilator extension of $(\mathcal{A}, \bullet_{\mathcal{A}})$ by $\mathfrak{h}^{*}$ by means of $\mu$, where
\[
\mu(u,v)(a) = -\omega_{\mathcal{A}}\big(G(a)u, v\big) \quad \text{for all } u,v \in \mathcal{A} \text{ and } a \in \mathfrak{h}.
\]

Furthermore, since $F(a)$ is symmetric with respect to $\omega_{\mathcal{A}}$ for all $a \in \mathfrak{h}$ and the compatibility system \eqref{les equation} is satisfied, it follows from Lemma~\ref{le2} that $(\mathfrak{h}, \widetilde{\mathcal{A}}, \widetilde{F}, \widetilde{G}, \widetilde{\theta})$ is itself an almost-semidirect product context of left-symmetric $\mathrm{L}$-algebras. Consequently, the vector space $\overline{\mathcal{A}} := \mathfrak{h} \oplus \mathcal{A} \oplus \mathfrak{h}^{*}$ equipped with the product $\bullet$ defined in \eqref{eq:prod-tildeA} is a left-symmetric $\mathrm{L}$-algebra.

Finally, one can easily verify that the bilinear form $\omega$ defines a pre-symplectic structure on $(\overline{\A},\bullet)$.
\end{proof}

\section{Some particular cases of double extensions of pre-symplectic left-symmetric algebras}

In this section, we study some particular cases of double extensions for which the defining equations simplify significantly.

Let $(\mathcal{A}, \bullet_\mathcal{A}, \omega_\mathcal{A})$ be a pre-symplectic left-symmetric algebra and let $(\mathfrak{h}, \mathcal{A}, F, G, \theta, L, \Omega)$ be a double extension context of $(\mathcal{A}, \bullet_\mathcal{A}, \omega_\mathcal{A})$ by a commutative associative algebra $(\mathfrak{h}, \circ)$.

\subsection{Double extensions by one-dimensional algebras}
In this case, $\dim \mathfrak{h} = 1$. We write $\mathfrak{h} := \mathbb{K} d$ and $\mathfrak{h}^{*} := \mathbb{K} e$, where $e(d) = 1$. Hence, the maps $F, G, \theta, L, \Omega$ and the product $\circ$ are determined by
\begin{equation}\label{Lesapp1}
F(d) = \xi, \quad G(d) = D, \quad \theta(d,d) = b_0, \quad L = \Omega = 0, \quad d \circ d = \lambda d,
\end{equation}
where $\xi \in \mathrm{End}(\mathcal{A})$ is symmetric with respect to $\omega_{\mathcal{A}}$, $D \in \mathrm{End}(\mathcal{A})$, $b_0 \in \mathcal{A}$, and $\lambda \in \mathbb{K}$.

The compatibility equations from \eqref{produit-semi} and \eqref{les equation} reduce precisely to the following system for all $u,v \in \mathcal{A}$:
\begin{equation}\label{Lesrelation dim1}
\begin{cases}
D([u,v]_{\bullet_\mathcal{A}}) = 0, \quad u \bullet_\mathcal{A} D(v) = v \bullet_\mathcal{A} D(u), \\[3pt]
D(u) \bullet_\mathcal{A} v = \xi(u) \bullet_\mathcal{A} v, \quad \xi(u \bullet_\mathcal{A} v) = u \bullet_\mathcal{A} \xi(v), \\[3pt]
\xi \circ D = \lambda D + \Rr^{\bullet_\mathcal{A}}_{b_0}, \quad D^2 = D \circ \xi, \quad D^* \circ D = 0, \quad D^*(b_0) = 0.
\end{cases}
\end{equation}

The following theorem summarizes the double extension of pre-symplectic left-symmetric algebras by one-dimensional algebras.

\sssbegin{Theorem}\label{double-ex1}
Let $(\mathcal{A}, \bullet_\mathcal{A}, \omega_{\mathcal{A}})$ be a pre-symplectic left-symmetric algebra. Let $\mathfrak{h} := \mathbb{K}d$ be a one-dimensional algebra and $\mathfrak{h}^* = \mathbb{K}e$ denote its dual. Assume that there exist two linear maps $\xi, D : \mathcal{A} \to \mathcal{A}$, an element $b_{0} \in \mathcal{A}$, and a scalar $\lambda \in \mathbb{K}$ such that $\xi$ is symmetric with respect to $\omega_{\mathcal{A}}$ and the system  \eqref{Lesrelation dim1} is satisfied. We define a bilinear product $\bullet$ on the vector space $\overline{\mathcal{A}} := \mathbb{K}d \oplus \mathcal{A} \oplus \mathbb{K}e$ by:
\begin{equation}\label{Produit1}
\begin{array}{lll} \displaystyle 
e \bullet u = u \bullet e=e \bullet d =0,& d \bullet e = \la e,& \displaystyle  
d \bullet u = \xi(u) -\om_\A( b_0, u) e, \\[0.3em]
\displaystyle  u \bullet d = D(u), &d \bullet d =  \la d+b_0, & 
u \bullet v = u \bullet_A v - \om_\A( D(u), v) e, 
\end{array}
\end{equation}
for all $u,v \in \mathcal{A}$, and a bilinear form $\omega$ given by:
\[
\omega|_{\mathcal{A} \times \mathcal{A}} = \omega_{\mathcal{A}}, \qquad \omega(e,d) =- \omega(d,e)= 1, \qquad \omega(d, \mathcal{A}) = \omega(e, \mathcal{A}) = \{0\}.
\]
Then $(\overline{\mathcal{A}}, \bullet, \omega)$ is a pre-symplectic left-symmetric algebra.
\end{Theorem}

The pre-symplectic left-symmetric algebra $(\overline{\mathcal{A}}, \bullet, \omega)$ is called the \emph{double extension} of the pre-symplectic left-symmetric algebra $(\mathcal{A}, \bullet_\mathcal{A}, \omega_{\mathcal{A}})$ by one-dimensional  algebra $\h$ by means of $(D, \xi, b_0, \la)$.

\subsection{Double extensions by two-dimensional commutative associative algebras}

In this case, $\dim \mathfrak{h} = 2$. We denote by $\mathfrak{h} := \mathbb{K} d_1 \oplus \mathbb{K} d_2$ and $\mathfrak{h}^{*} := \mathbb{K} e_1 \oplus \mathbb{K} e_2$, where $e_i(d_i) = 1$ and $e_i(d_j) = 0$ for $i \neq j$. Thus, the maps $F, G, \theta, L, \Omega$ and the product $\circ$ are defined by:
\begin{equation}\label{prdim2}
\begin{aligned}
F(d_i) &= \xi_i, \quad G(d_i) = D_i, \quad \theta(d_i, d_j) = b_{ij}, \quad \Omega(d_1,d_2) = -\Omega(d_2, d_1) = a_0, \quad \Omega(d_i,d_i) = 0, \\
L(d_1,d_1) &= l_1e_2, \quad L(d_1,d_2) = -l_{1}e_1, \quad L(d_2,d_1) = l_2e_2, \quad L(d_2,d_2) = -l_{2}e_1, \quad
d_i \circ d_j = t_{ij} d_1 + s_{ij}d_2,
\end{aligned}
\end{equation}
where $i,j \in \{1,2\}$, $\xi_i \in \mathrm{End}(\mathcal{A})$ are symmetric with respect to $\omega_{\mathcal{A}}$, $D_i \in \mathrm{End}(\mathcal{A})$, $a_0, b_{ij} \in \mathcal{A}$, and $l_i, t_{ij}, s_{ij} \in \mathbb{K}$ with $t_{ij}=t_{ji}$ and $s_{ij}=s_{ji}$.

The compatibility equations from \eqref{produit-semi} and \eqref{les equation} reduce precisely to the following system for all $u,v \in \mathcal{A}$ and $i,j,k \in \{1,2\}$:
\begin{equation}\label{Les relation dim2}
\begin{cases}
D_i([u,v]_{\bullet_\mathcal{A}}) = 0, \quad u \bullet_\mathcal{A} D_i(v) = v \bullet_\mathcal{A} D_i(u), \quad D_i(u) \bullet_\mathcal{A} v = \xi_i(u) \bullet_\mathcal{A} v, \\[3pt]
\xi_i(u \bullet_\mathcal{A} v) = u \bullet_\mathcal{A} \xi_i(v), \quad D_i \circ D_j = D_i \circ \xi_j, \quad \xi_i \circ D_j = \xi_j \circ \xi_i, \\[3pt]
\xi_i \circ D_j(u) = t_{ij} D_1(u) + s_{ij}D_2(u) + \Rr_{b_{ij}}^{\bullet_\mathcal{A}}(u), \quad D_i^* \circ D_j = -D_j^* \circ D_i, \quad \Ll_{b_{ij}-b_{ji}}^{\bullet_\mathcal{A}} = 0, \\[3pt]
D_k(b_{ij} - b_{ji}) = 0, \quad \xi_j(b_{ik}) - \xi_i(b_{jk}) = t_{jk}b_{i1} + s_{jk}b_{i2} - t_{ik}b_{j1} - s_{ik}b_{j2}, \\[3pt]
(D_i^* - \xi_i)(a_0) = 0,  \quad D_1^*(b_{j1}) = s_{1j}a_0, \quad
D_2^*(b_{j2}) =- t_{2j}a_0, \\[3pt] D_1^*(b_{i2}) - s_{i2}a_0 = -D_2^*(b_{i1}) - t_{i1}a_0, 
l_1s_{12} - l_2s_{11} = -\omega_\A(b_{11}, b_{21}), \\[3pt]  l_1t_{22} - l_2t_{12} = \omega_\A(b_{12}, b_{22}), \quad l_1(t_{12} - s_{22}) + l_2(s_{12} - t_{11}) = \omega_\A(b_{21}, b_{12})-\omega_\A(b_{11}, b_{21}), \\[3pt]
\omega_\mathcal{A}\big(a_0, [u,v]_{\bullet_\mathcal{A}}\big) = 0, \quad \omega_\mathcal{A}(a_0, b_{ij} - b_{ji}) = 0.
\end{cases}
\end{equation}

Therefore, we have the following theorem.

\sssbegin{Theorem}\label{double-extension2}
Let $(\mathcal{A}, \bullet_\mathcal{A}, \omega_{\mathcal{A}})$ be a pre-symplectic left-symmetric algebra and let $(\mathfrak{h}, \circ)$ be a two-dimensional commutative associative algebra. Denote by $\mathfrak{h} := \mathbb{K}d_1 \oplus \mathbb{K}d_2$ and $\mathfrak{h}^{*} := \mathbb{K}e_1 \oplus \mathbb{K}e_2$, where $e_i(d_i) = 1$ and $e_i(d_j) = 0$ for $i \neq j$, $i,j \in \{1,2\}$.

Assume that there exist linear maps $\xi_i, D_i : \mathcal{A} \to \mathcal{A}$, elements $b_{ij}, a_0 \in \mathcal{A}$, and scalars $l_i, t_{ij}, s_{ij} \in \mathbb{K}$ such that each $\xi_i$ is symmetric with respect to $\omega_{\mathcal{A}}$, $t_{ij} = t_{ji}$, $s_{ij} = s_{ji}$, and the system \eqref{Les relation dim2} is satisfied. We define a bilinear product $\bullet$ on the vector space $\overline{\mathcal{A}} := \mathfrak{h} \oplus \mathcal{A} \oplus \mathfrak{h}^*$ defined by:
\begin{equation}\label{Produit de dim2}
\begin{cases} 
d_1 \bullet u = \xi_1(u) - \omega_{\mathcal{A}}(b_{11}, u)e_1 - \omega_{\mathcal{A}}(b_{12}, u)e_2, \\[3pt]
d_2 \bullet u = \xi_2(u) - \omega_{\mathcal{A}}(b_{21}, u)e_1 - \omega_{\mathcal{A}}(b_{22}, u)e_2, \\[3pt]
u \bullet d_1 = D_1(u) + \omega_{\mathcal{A}}(a_0, u)e_2, \\[3pt]
u \bullet d_2 = D_2(u) - \omega_{\mathcal{A}}(a_0, u)e_1, \\[3pt]
u \bullet v = u \bullet_\mathcal{A} v - \omega_{\mathcal{A}}(D_1(u), v)e_1 - \omega_{\mathcal{A}}(D_2(u), v)e_2, \\[3pt]
d_i \bullet e_1 = t_{i1} e_1 + t_{i2}e_2, \\[3pt]
d_i \bullet e_2 = s_{i1} e_1 + s_{i2}e_2, \\[3pt]
d_1 \bullet d_1 = b_{11} + l_1 e_2 + t_{11}d_1 + s_{11}d_2, \\[3pt]
d_1 \bullet d_2 = b_{12} - l_1 e_1 + t_{12}d_1 + s_{12}d_2, \\[3pt]
d_2 \bullet d_1 = b_{21} + l_2 e_2 + t_{12}d_1 + s_{12}d_2, \\[3pt]
d_2 \bullet d_2 = b_{22} - l_2 e_1 + t_{22}d_1 + s_{22}d_2,
\end{cases}
\end{equation}
for all $u,v \in \mathcal{A}$ and $i \in \{1,2\}$, with all products not listed above are zero and a bilinear form $\omega$ given by:
\[
\omega|_{\mathcal{A} \times \mathcal{A}} = \omega_{\mathcal{A}}, \qquad \omega(e_i, d_i)=-\omega(d_i, e_i) = 1,
\]
for $i \in \{1,2\}$, with all other cross-terms being equal to zero.

Then $(\overline{\mathcal{A}}, \bullet, \omega)$ is a pre-symplectic left-symmetric algebra.
\end{Theorem}

The pre-symplectic left-symmetric algebra $(\overline{\mathcal{A}}, \bullet, \omega)$ is called the \emph{double extension} of the pre-symplectic left-symmetric algebra $(\mathcal{A}, \bullet_\mathcal{A}, \omega_{\mathcal{A}})$ by the two-dimensional commutative associative algebra $(\mathfrak{h}, \circ)$ by means of $(D_i, \xi_i, a_0, b_{ij}, l_i, t_{ij}, s_{ij})$ for $i,j \in \{1,2\}$.

\section{Pre-symplectic left-symmetric algebras
with degenerate commutator ideal}

In this section, we establish the converse of Theorems~\ref{Double-extension}, \ref{double-ex1}, and \ref{double-extension2}, and we characterize pre-symplectic left-symmetric algebras in terms of the notion of double extension.

\ssbegin{Proposition}\label{reduit1}
Let $(\mathcal{A}, \bullet, \omega)$ be a pre-symplectic left-symmetric algebra. Let $I \subseteq [\mathcal{A}, \mathcal{A}] \cap [\mathcal{A}, \mathcal{A}]^\perp$ be a two-sided ideal of $(\mathcal{A}, \bullet)$. Denote by $\mathcal{B} := I^\perp / I$ and $\mathfrak{h} := \mathcal{A} / I^\perp$, and let $\pi_{\mathcal{B}} : I^\perp \to \mathcal{B}$ and $\pi_{\mathfrak{h}} : \mathcal{A} \to \mathfrak{h}$ be the canonical projections. Then:
\begin{enumerate}
    \item $I^\perp$ is a two-sided ideal of $(\mathcal{A}, \bullet)$. Moreover, we have $I^\perp \bullet I = I \bullet \mathcal{A} = \{0\}$.
    
    \item $\mathcal{B}$ inherits a canonical structure of a pre-symplectic left-symmetric algebra, defined by:
    \[
    \pi_\mathcal{B}(u) \bullet_{\mathcal{B}} \pi_\mathcal{B}(v) := \pi_\mathcal{B}(u \bullet v), \qquad \omega_{\mathcal{B}}( \pi_\mathcal{B}(u), \pi_\mathcal{B}(v) ) := \omega(u, v),
    \]
    for all $u, v \in I^\perp$.
    
    \item $\mathfrak{h}$ inherits a canonical structure of a commutative associative algebra, defined by:
    \[
    \pi_{\mathfrak{h}}(u) \circ \pi_{\mathfrak{h}}(v) := \pi_{\mathfrak{h}}(u \bullet v),
    \]
    for all $u, v \in \mathcal{A}$. 
\end{enumerate}
\end{Proposition}

\begin{proof}
\begin{enumerate}
\item Let $u \in I$, $v \in I^\perp$, and $w \in \mathcal{A}$. Then, we have:
\[
\omega(w \bullet v, u) = \omega(v, w \bullet u) = 0,
\]
since $I$ is a two-sided ideal of $(\mathcal{A}, \bullet)$, which implies that $I^\perp$ is a left ideal of $(\mathcal{A}, \bullet)$. Moreover, since $u \in I \subseteq [\mathcal{A},\mathcal{A}] \cap [\mathcal{A},\mathcal{A}]^\perp$, equation \eqref{ortogonal} implies that $\Rr^\bullet_u = -(\Rr^\bullet_u)^*$. Thus:
\[
\omega(v \bullet w, u) = \omega(w, v \bullet u) = -\omega(w \bullet u, v) = 0,
\]
which shows that $I^\perp$ is also a right ideal of $(\mathcal{A}, \bullet)$ and $I^\perp \bullet I = \{0\}$. Hence, $I^\perp$ is a two-sided ideal of $(\mathcal{A}, \bullet)$. Furthermore, since $I \subseteq [\mathcal{A},\mathcal{A}] \cap [\mathcal{A},\mathcal{A}]^\perp$, according to Proposition~\ref{condition plate}, it follows that $\Ll^\bullet_u = 0$ for all $u \in I$, and hence $I \bullet \mathcal{A} = \{0\}$.

\item Since $I^\perp$ is a two-sided ideal of $(\A, \bullet)$, it follows that the operations:
\[
\pi_{\mathcal{B}}(u) \bullet_{\mathcal{B}} \pi_{\mathcal{B}}(v) := \pi_{\mathcal{B}}(u \bullet v) \quad \text{and} \quad \omega_{\mathcal{B}}(\pi_{\mathcal{B}}(u), \pi_{\mathcal{B}}(v)) := \omega(u,v)
\]
for all $u,v \in I^\perp$ are well defined. It is straightforward to verify that $(\mathcal{B}, \bullet_{\mathcal{B}}, \omega_{\mathcal{B}})$ is a pre-symplectic left-symmetric algebra.

\item Clearly, the operation $\pi_{\mathfrak{h}}(u) \circ \pi_{\mathfrak{h}}(v) := \pi_{\mathfrak{h}}(u \bullet v)$ for all $u, v \in \mathcal{A}$ defines a left-symmetric algebra structure on $\mathfrak{h}$. Since $[\mathcal{A}, \mathcal{A}] \subseteq I^\perp$, we obtain:
\[
\pi_{\mathfrak{h}}(u) \circ \pi_{\mathfrak{h}}(v) = \pi_{\mathfrak{h}}(v) \circ \pi_{\mathfrak{h}}(u)
\]
for all $u,v \in \mathcal{A}$, which proves that $\mathfrak{h}$ is indeed a commutative associative algebra.
\end{enumerate}
\end{proof}

\ssbegin{Theorem}\label{Lemme-ex}
Let $(\mathcal{A}, \bullet, \omega)$ be a pre-symplectic left-symmetric algebra. Assume that there exists a two-sided ideal  $I \subseteq [\mathcal{A}, \mathcal{A}] \cap [\mathcal{A}, \mathcal{A}]^\perp$ of $(\mathcal{A}, \bullet)$. Then $(\mathcal{A}, \bullet, \omega)$ is a double extension of the pre-symplectic left-symmetric algebra $(\mathcal{B} := I^\perp / I, \bullet_\mathcal{B}, \omega_\mathcal{B})$ by the commutative associative algebra $\mathfrak{h} := (\mathcal{A} / I^\perp, \circ)$, by means of $(F, G, \theta, L, \Omega)$.
\end{Theorem}

\begin{proof}
Since $I \subseteq [\mathcal{A}, \mathcal{A}] \cap [\mathcal{A}, \mathcal{A}]^\perp$ is a totally isotropic two-sided ideal of $(\mathcal{A}, \bullet)$, Proposition~\ref{reduit1} implies that $I^\perp$ is also a two-sided ideal of $(\mathcal{A}, \bullet)$, and $I^\perp \bullet I = I \bullet \mathcal{A} = \{0\}$. Moreover, $\mathcal{B} := I^\perp / I$ inherits a structure of a pre-symplectic left-symmetric algebra, and $\mathfrak{h} := \mathcal{A} / I^\perp$ inherits a structure of a commutative associative algebra.

Now, since $I \subseteq [\mathcal{A}, \mathcal{A}] \cap [\mathcal{A}, \mathcal{A}]^\perp$ is totally isotropic, we can choose a vector subspace $B \subseteq I^\perp$ such that $I^\perp = B \oplus I$. There exists a totally isotropic vector subspace $V \subseteq \mathcal{A}$ such that $B^\perp = I \oplus V$. The restriction $\omega_B$ of $\omega$ to $B$ is nondegenerate, and the map $I \longrightarrow V^*$, given by $u \mapsto \omega(u, \cdot)|_V$, is an isomorphism of vector spaces. Consequently, the pre-symplectic vector space $(B^\perp, \omega|_{B^\perp \times B^\perp})$ is isometric to $V \oplus V^{*}$ endowed with the canonical pre-symplectic  bilinear form:
\[
\omega_0(a + f, b + g) = f(b) - g(a) \quad \text{for all } a, b \in V, \; f,g \in V^{*}.
\]
Hence, we obtain the vector space decomposition:
$
(\mathcal{A}, \omega) \simeq (V \oplus B \oplus V^*, \; \omega_n = \omega_0 + \omega_B).$

Since both $I$ and $I^\perp$ are two-sided ideals of $(\mathcal{A}, \bullet)$ and $I \bullet \mathcal{A} = \{0\}$, the bilinear product $\bullet$ can be written  as:

\begin{equation}\label{decomposition-product}
 \begin{cases}
u\bullet v = u\bullet_B v +\mu(u,v),\\
u\bullet a = G(a)u + \tau(u,a),\\
a\bullet u = F(a)u +\varphi(a, u),\\
a\bullet b = a\circ_V b +\theta(a,b)+ L(a,b),\\
a\bullet f = N(a)(f),
\\ \Ll^\bullet_f=0,
\end{cases}
\end{equation}
for all $a, b \in V$, $u,v \in B$, and $f \in V^*$, where $a \circ_V b \in V$, $\theta(a, b), u \bullet_B v \in B$, and $L(a, b), \varphi(a,u), \tau(u,a), N(a)(f) \in V^*$, with $F(a), G(a) \in \mathrm{End}(B)$.

Since $(\mathcal{A}, \bullet, \omega)$ is a pre-symplectic left-symmetric algebra, it follows that $(B, \bullet_B)$ is a left-symmetric algebra and that the restriction $\omega_B := \omega|_{B \times B}$ is a pre-symplectic structure on $(B, \bullet_B)$. Therefore, $\mu$ satisfies Condition \eqref{con-left}. The canonical projection $\pi_\mathcal{B} : I^\perp \to \mathcal{B}$ identifies $(B, \bullet_B, \omega_B)$ with $(\mathcal{B}, \bullet_\mathcal{B}, \omega_\mathcal{B})$.

Moreover, since $(\mathcal{A}, \bullet)$ is a left-symmetric algebra and $[\mathcal{A}, \mathcal{A}] \subseteq I^\perp$, we deduce that $[a,b]_\bullet \in I^\perp$ for all $a, b \in V$, which implies:
\[
a \circ_V b = b \circ_V a \quad \text{for all } a,b \in V.
\]
Thus, $(V, \circ_V)$ is a commutative associative algebra. Moreover, the canonical projection $\pi_{\mathfrak{h}} : \mathcal{A} \to \mathfrak{h}$ identifies $(V, \circ_V)$ with $(\mathfrak{h}, \circ)$. Consequently, $(\mathcal{A}, \bullet, \omega)$ can be identified with $(\mathfrak{h} \oplus \mathcal{B} \oplus \mathfrak{h}^*, \bullet, \omega_n)$.

Using the fact that left multiplications in $(\mathcal{A}, \bullet)$ are symmetric with respect to $\omega$, we obtain the following explicit relations. First, for all $u,v \in \mathcal{B}$ and $a \in \mathfrak{h}$, we have:
\[
\mu(u,v)(a) = \omega_n(u \bullet v, a) = \omega(v, u \bullet a) = \omega_{\mathcal{B}}(v, G(a)u) = -\omega_{\mathcal{B}}(G(a)u, v),
\]
which implies that $\mu(u,v) = -\omega_{\mathcal{B}}(G(\cdot)u, v)$.

Next, for all $a,b \in \mathfrak{h}$ and $u \in \mathcal{B}$:
\[
\tau(u,a)(b) = \omega_n(u \bullet a, b) = \omega_n(a, u \bullet b) = -\tau(u,b)(a).
\]
Thus, $\tau(u,a)(b) = - \tau(u,b)(a)$. Since $\omega_{\mathcal{B}}$ is nondegenerate, there exists a unique  bilinear map $\Omega : \mathfrak{h} \times \mathfrak{h} \to \mathcal{B}$ such that:
\[
\tau(u,a)(b) = \omega_{\mathcal{B}}(\Omega(a, b), u) = -\omega_{\mathcal{B}}(\Omega(b, a), u),
\]
and consequently $\Omega(a,b) = -\Omega(b,a)$.

For any $u,v \in \mathcal{B}$ and $a \in \mathfrak{h}$, we obtain:
\[
\omega_{\mathcal{B}}(F(a)u, v) = \omega_n(a \bullet u, v) = \omega_n(u, a \bullet v) = \omega_{\mathcal{B}}(u, F(a)v),
\]
showing that $F(a)$ is symmetric with respect to $\omega_{\mathcal{B}}$.

Similarly, for all $u \in \mathcal{B}$ and $a,b \in \mathfrak{h}$, we have:
\[
\varphi(a,u)(b) = \omega_n(a \bullet u, b) = \omega_n(u, a \bullet b) = \omega_{\mathcal{B}}(u, \theta(a,b)),
\]
which implies $\varphi(a,u) = \omega_{\mathcal{B}}(u, \theta(a,\cdot)) = -\omega_{\mathcal{B}}(\theta(a,\cdot), u)$.

Next, for all $a,b \in \mathfrak{h}$ and $f \in \mathfrak{h}^*$, we have:
\[
f(a \circ b) = -\omega_n(a \circ b, f) = -\omega_n(b, a \bullet f) = -\omega_n(b, N(a)(f)) = N(a)(f)(b),
\]
which yields $N(a)(f) = f \circ \Ll_a^{\circ}$.

Finally, for all $a,b,c\in \h$, we obtain
\[
L(a,b)(c)
= \omega_n(L(a,b),c)
= \omega(a\bullet b,c)
= \omega_n(b, a\bullet c)
= \omega(b, L(a,c))
= -L(a,c)(b),
\]
showing that
\[
L(a,b)(c) = -L(a,c)(b).
\]
The product \eqref{decomposition-product} can now be rewritten in an equivalent form.
\begin{equation}\label{decomposition-product1}
    \begin{cases}
u\bullet v = u\bullet_\mathcal{B} v - \om_\mathcal{B}( G(\cdot)u,v),\\
u\bullet a = G(a)u + \om_\mathcal{B}( \Om(a,\cdot),u),\\
a\bullet u = F(a)u - \om_\mathcal{B}( \theta(a,\cdot),u),\\
a\bullet b = a\circ b +\theta(a,b)+ L(a,b),\\
a\bullet f = f\circ \Ll_a^\circ,
\end{cases}
\end{equation}
for all $u,v\in \mathcal{B}$, $a,b\in \h$  and $f\in \h^*$, where $F(a)$ is symmetric with respect to $\om_{\mathcal{B}}$ and $L(a,b)(c) =-L(a,c)(b)$ for all $a,b,c \in \h$.

By direct computation, one verifies that the triple $(F,G,\theta)$ is  an almost-semidirect product context of left-symmetric $\Ll$-algebras, and that $(F,G,\theta,L,\Omega)$ satisfies the compatibility relations in \eqref{les equation}. Consequently, $(\mathcal{A}, \bullet, \omega)$ is a double extension of the pre-symplectic left-symmetric algebra $(\mathcal{B} = I^\perp / I, \bullet_\mathcal{B}, \omega_\mathcal{B})$ by the commutative associative algebra $\mathfrak{h} = (\mathcal{A} / I^\perp, \circ)$, by means of $(F, G, \theta, L, \Omega)$.
\end{proof}

\ssbegin{Theorem}\label{Double-generale}
Let $(\mathcal{A}, \bullet, \omega)$ be a pre-symplectic left-symmetric algebra such that the commutator ideal $[\mathcal{A}, \mathcal{A}]$ is degenerate with respect to $\omega$. Set
$
I := [\mathcal{A}, \mathcal{A}] \cap [\mathcal{A}, \mathcal{A}]^\perp.
$
Then $(\mathcal{A}, \bullet, \omega)$ is a double extension of the pre-symplectic left-symmetric algebra $(\mathcal{B} := I^\perp / I, \bullet_\mathcal{B}, \omega_\mathcal{B})$ by the commutative associative algebra $(\mathfrak{h} := \mathcal{A} / I^\perp, \circ)$ by means of $(F, G, \theta, L, \Omega)$.
\end{Theorem}

\begin{proof}
Assume that $I \neq \{0\}$. By Proposition~\ref{pr 6.7}, $I$ is a totally isotropic two-sided ideal of $(\mathcal{A}, \bullet)$, and its orthogonal complement $I^\perp$ is also a two-sided ideal of $(\mathcal{A}, \bullet)$. Hence, by applying Theorem~\ref{Lemme-ex}, it follows directly that $(\mathcal{A}, \bullet, \omega)$ is a double extension of the pre-symplectic left-symmetric algebra $(\mathcal{B} := I^\perp / I, \bullet_\mathcal{B}, \omega_\mathcal{B})$ by the commutative associative algebra $(\mathfrak{h} := \mathcal{A} / I^\perp, \circ)$ by means of $(F, G, \theta, L, \Omega)$.
\end{proof}

\ssbegin{Theorem}\label{Milnor0}
Any pre-symplectic left-symmetric algebra   is either a Milnor pre-symplectic algebra or can be obtained by a finite sequence of double extensions by a  commutative associative algebra starting from a Milnor pre-symplectic algebra.
\end{Theorem}

\begin{proof}
If $[\mathcal{A}, \mathcal{A}]$ is nondegenerate with respect to $\omega$, then according to Theorem~\ref{ ideal non degenere}, $(\mathcal{A}, \bullet, \omega)$ is a Milnor pre-symplectic algebra. If $[\mathcal{A}, \mathcal{A}]$ is degenerate, then according to Theorem~\ref{Double-generale}, $(\mathcal{A}, \bullet, \omega)$ is a double extension of a lower-dimensional pre-symplectic left-symmetric algebra $(\mathcal{B}_1, \bullet_{\mathcal{B}_1}, \omega_{\mathcal{B}_1})$ by a  commutative associative algebra $(\mathfrak{h}_1, \circ_1)$.

If the commutator ideal $[\mathcal{B}_1, \mathcal{B}_1]$ is nondegenerate with respect to $\omega_{\mathcal{B}_1}$, then $(\mathcal{B}_1, \bullet_{\mathcal{B}_1}, \omega_{\mathcal{B}_1})$ is a Milnor pre-symplectic algebra and the reduction stops. If it is degenerate, then $(\mathcal{B}_1, \bullet_{\mathcal{B}_1}, \omega_{\mathcal{B}_1})$ is itself a double extension of a pre-symplectic left-symmetric algebra $(\mathcal{B}_2, \bullet_{\mathcal{B}_2}, \omega_{\mathcal{B}_2})$ by a  commutative associative algebra $(\mathfrak{h}_2, \circ_2)$.

Since the vector space $\mathcal{A}$ is assumed to be finite-dimensional, the dimensions of the successive quotients decrease strictly at each step ($\dim \mathcal{B}_{k+1} < \dim \mathcal{B}_k$). Consequently, this algorithmic reduction process must terminate after finitely many steps. At the final step $n$, the core algebra $(\mathcal{B}_n, \bullet_{\mathcal{B}_n}, \omega_{\mathcal{B}_n})$ will necessarily have a nondegenerate commutator ideal, and thus constitutes a Milnor pre-symplectic algebra.
\end{proof}

Let $(\mathcal{A}, \bullet)$ be a left-symmetric algebra. It is well known that the adjoint map
\[
\mathrm{ad} : \mathcal{A} \longrightarrow \mathrm{End}(\mathcal{A}), \qquad \mathrm{ad}_u(v) = [u,v],
\]
defines a representation of the sub-adjacent Lie algebra $\mathcal{A}^{-}$.

\ssbegin{Lemma}\label{Le3}
Let $(\mathcal{A}, \bullet, \omega)$ be a pre-symplectic left-symmetric algebra over a field $\mathbb{K}$ (where $\mathbb{K}=\mathbb{R}$, $\mathbb{C}$, or $\mathbb{K}$ is algebraically closed) such that the commutator ideal $[\mathcal{A}, \mathcal{A}]$ is degenerate with respect to $\omega$.

\begin{enumerate}
    \item[\rm $(i)$] If $\mathbb{K}$ is algebraically closed, then there exists a $1$-dimensional totally isotropic two-sided ideal $J$ of $(\mathcal{A}, \bullet)$ such that $J \subseteq [\mathcal{A}, \mathcal{A}] \cap [\mathcal{A}, \mathcal{A}]^\perp$. Moreover, its orthogonal complement $J^\perp$ is also a two-sided ideal of $(\mathcal{A}, \bullet)$.
    \item[\rm $(ii)$] If $\mathbb{K}=\mathbb{R}$, then there exists a totally isotropic two-sided ideal $J$ of $(\mathcal{A}, \bullet)$ of dimension $1$ or $2$ such that $J \subseteq [\mathcal{A}, \mathcal{A}] \cap [\mathcal{A}, \mathcal{A}]^\perp$. Moreover, its orthogonal complement $J^\perp$ is also a two-sided ideal of $(\mathcal{A}, \bullet)$.
\end{enumerate}
\end{Lemma}

\begin{proof}
According to Proposition~\ref{condition plate}, we have
$\Ll_{[u,v]_\bullet} = 0$ for all $u,v \in \mathcal{A}$. Hence, the sub-adjacent Lie algebra $\mathcal{A}^-$ is solvable.
By Proposition~\ref{pr 6.7}, the space $
I:=[\mathcal{A},\mathcal{A}]
\cap[\mathcal{A},\mathcal{A}]^\perp
$
is a two-sided ideal of $(\mathcal{A},\bullet)$.
Since $[\mathcal{A},\mathcal{A}]$ is degenerate
with respect to $\omega$, we have $I\neq\{0\}$.

Therefore, the restriction of the adjoint representation
$
\mathrm{ad}:\mathcal{A}^-
\longrightarrow \mathrm{End}(I)
$
defines a representation of the solvable Lie algebra
$\mathcal{A}^-$.

Suppose first that $\mathbb{K}$ is algebraically closed.
By Lie's theorem, there exists a nonzero vector
$x\in I$ such that
\[
\mathrm{ad}_u(x)=\lambda(u)x,
\qquad \text{for all }u\in\mathcal{A},
\]
where $\lambda:\mathcal{A}\to\mathbb{K}$ is a linear map.
Thus, $J=\mathbb{K}x$ is a one-dimensional
Lie ideal contained in $I$.

Now suppose that $\mathbb{K}=\mathbb{R}$.
Consider the complexifications
$
\mathcal{A}_{\mathbb{C}}
=\mathcal{A}\otimes_{\mathbb{R}}\mathbb{C}$ and 
$I_{\mathbb{C}}
=I\otimes_{\mathbb{R}}\mathbb{C}$. Since $\mathcal{A}^-$ is solvable, its complexification
$\mathcal{A}_{\mathbb{C}}^-$ is also solvable.
By Lie's theorem, there exists a nonzero vector
$z\in I_{\mathbb{C}}$ and a complex linear form
$\lambda:\mathcal{A}_{\mathbb{C}}\to\mathbb{C}$
such that
\[
\mathrm{ad}_u(z)=\lambda(u)z,
\qquad \text{for all }u\in\mathcal{A}.
\]
Writing $z=x+iy$, where $x,y\in I$, and
$\lambda(u)=a(u)+ib(u)$ for $u\in\mathcal{A}$,
we obtain
\[
[u,x]=a(u)x-b(u)y,
\qquad
[u,y]=b(u)x+a(u)y.
\]
Consequently, $
J=\mathrm{span}_{\mathbb{R}}\{x,y\}$ 
is a nonzero Lie ideal of $\mathcal{A}^-$
contained in $I$, with
$\dim_{\mathbb{R}}J\in\{1,2\}$.

In both cases, we have obtained a nonzero
Lie ideal $J\subseteq I$ of the required dimension.
Since
$
J\subseteq[\mathcal{A},\mathcal{A}]
\cap[\mathcal{A},\mathcal{A}]^\perp,
$ the ideal $J$ is totally isotropic.

Moreover, Proposition~\ref{condition plate} implies that $\Ll^\bullet_u = 0$ for all $u \in J$, which means $u \bullet v = 0$ for all $v \in \mathcal{A}$. Consequently, $J$ is a right ideal of $(\mathcal{A}, \bullet)$. Furthermore, for all $u \in J$ and $v \in \mathcal{A}$, the condition $[\mathcal{A}, J] \subseteq J$ implies that $[v, u] = v \bullet u - u \bullet v = v \bullet u \in J$. This shows that $J$ is also a left ideal of $(\mathcal{A}, \bullet)$, and thus it constitutes a two-sided ideal. 

Finally, since $J \subseteq [\mathcal{A}, \mathcal{A}] \cap [\mathcal{A}, \mathcal{A}]^\perp$ is a two-sided ideal of $(\mathcal{A}, \bullet)$, Proposition~\ref{reduit1} ensures that its orthogonal complement $J^\perp$ is also a two-sided ideal of $(\mathcal{A}, \bullet)$.
\end{proof}

\ssbegin{Theorem}\label{Closed}
Let $(\mathcal{A}, \bullet, \omega)$ be a pre-symplectic left-symmetric algebra over an algebraically closed field $\mathbb{K}$ such that the commutator ideal $[\mathcal{A}, \mathcal{A}]$ is degenerate with respect to $\omega$. Then there exists a $1$-dimensional totally isotropic two-sided ideal $I$ of $(\mathcal{A}, \bullet)$ such that $I \subseteq [\mathcal{A}, \mathcal{A}] \cap [\mathcal{A}, \mathcal{A}]^\perp$. Moreover, $(\mathcal{A}, \bullet, \omega)$ is a double extension of $(\mathcal{B} := I^\perp / I, \bullet_\mathcal{B}, \omega_\mathcal{B})$ by the $1$-dimensional algebra $\mathfrak{h} := \mathcal{A}/I^\perp \simeq \mathbb{K} d$, by means of $(D, \xi, b_0, \la)$.
\end{Theorem}

\begin{proof}
Since $[\mathcal{A}, \mathcal{A}]$ is degenerate with respect to $\omega$, Lemma~\ref{Le3} ensures the existence of a $1$-dimensional totally isotropic two-sided ideal $I \subseteq [\mathcal{A}, \mathcal{A}] \cap [\mathcal{A}, \mathcal{A}]^\perp$. According to Theorem~\ref{Lemme-ex}, the algebra $(\mathcal{A}, \bullet, \omega)$ is a double extension of $(\mathcal{B}, \bullet_\mathcal{B}, \omega_\mathcal{B})$ by the $1$-dimensional  algebra $\mathfrak{h} \simeq \mathbb{K} d$, by means of $(D, \xi, b_0, \lambda)$.
\end{proof}

\ssbegin{Theorem}\label{Milnor10}
Any pre-symplectic left-symmetric algebra over an algebraically closed field $\mathbb{K}$ is either a Milnor pre-symplectic algebra or can be obtained by a finite sequence of double extensions by a one-dimensional commutative associative algebra starting from a Milnor pre-symplectic algebra.
\end{Theorem}

\begin{proof}
The proof follows analogously to the proof of Theorem \ref{Milnor0}.
\end{proof}

\ssbegin{Theorem}\label{db-1-2}
Let $(\mathcal{A}, \bullet, \omega)$ be a pre-symplectic left-symmetric algebra over the field $\mathbb{R}$ such that the commutator ideal $[\mathcal{A}, \mathcal{A}]$ is degenerate with respect to $\omega$. Then there exists a totally isotropic two-sided ideal $I \subseteq [\mathcal{A}, \mathcal{A}] \cap [\mathcal{A}, \mathcal{A}]^\perp$ of dimension $1$ or $2$. Moreover, $(\mathcal{A}, \bullet, \omega)$ is either:
\begin{itemize}
    \item A double extension of $(\mathcal{B} := I^\perp / I, \bullet_\mathcal{B}, \omega_\mathcal{B})$ by the one-dimensional  algebra $\mathfrak{h} := \mathcal{A}/I^\perp \simeq \mathbb{R} d$, by means of $(D, \xi, b_0, \la)$; or

    \item A double extension of $(\mathcal{B} := I^\perp / I, \bullet_\mathcal{B}, \omega_\mathcal{B})$ by the two-dimensional commutative associative algebra $\mathfrak{h} := \mathcal{A}/I^\perp \simeq \mathbb{R} d_1 \oplus \mathbb{R} d_2$, by means of $(D_i, \xi_i, a_{0}, b_{ij}, l_i, t_{ij}, s_{ij})$ for $i,j \in \{1,2\}$.
\end{itemize}
\end{Theorem}

\begin{proof}
Since $[\mathcal{A}, \mathcal{A}]$ is degenerate with respect to $\omega$, Lemma~\ref{Le3} ensures the existence of a totally isotropic two-sided ideal $I \subseteq [\mathcal{A}, \mathcal{A}] \cap [\mathcal{A}, \mathcal{A}]^\perp$ of dimension $1$ or $2$. According to Theorem~\ref{Lemme-ex}, we distinguish two geometric cases depending on the dimension of $I$.

If $\dim I = 1$, then $(\mathcal{A}, \bullet, \omega)$ is a double extension of $(\mathcal{B}, \bullet_\mathcal{B}, \omega_\mathcal{B})$ by the one-dimensional commutative associative algebra $\mathfrak{h} \simeq \mathbb{R} d$, by means of $(D, \xi, b_0, \lambda)$.

If $\dim I = 2$, then $(\mathcal{A}, \bullet, \omega)$ is a double extension of $(\mathcal{B}, \bullet_\mathcal{B}, \omega_\mathcal{B})$ by the two-dimensional commutative associative algebra $\mathfrak{h} \simeq \mathbb{R} d_1 \oplus \mathbb{R} d_2$, by means of $(D_i, \xi_i, a_{0}, b_{ij}, l_i, t_{ij}, s_{ij})$ for $i,j \in \{1,2\}$.
\end{proof}

\ssbegin{Theorem}\label{Milnor1}
Any pre-symplectic left-symmetric algebra  over the field $\mathbb{R}$ is either a Milnor pre-symplectic algebra or can be obtained by a finite sequence of double extensions by a one-dimensional or two-dimensional commutative associative algebra starting from a Milnor pre-symplectic algebra.
\end{Theorem}

\begin{proof}
The proof follows analogously to the proof of Theorem \ref{Milnor0}.
\end{proof}

\section{Classification of pre-symplectic left-symmetric algebras of dimension $\leq 4$}\label{section4}
In this section, we apply Theorems~\ref{ ideal non degenere} and  \ref{Double-generale} to classify all pre-symplectic left-symmetric algebras of dimension up to $4$.

Let us start with the determination of pre-symplectic left-symmetric algebras of dimension $2$.
\ssbegin{Proposition}\label{mil3}
Let $(\mathcal{A}, \bullet, \omega)$ be a two-dimensional pre-symplectic left-symmetric algebra. Then $(\mathcal{A}, \bullet, \omega)$ is symplectic-isomorphic to one of the following mutually non-isomorphic algebras:
\begin{enumerate}
    \item $\mathbb{K}^2$ with the trivial product and $\omega = e_{1}^{*} \wedge e_{2}^{*}$.
    \item $(\A_2,\om)$: $e_1\bullet e_1=e_1$,   $e_1\bullet e_2=e_2$, $\om=e_{1}^{*}\wedge e_{2}^{*}$.
\end{enumerate}
\end{Proposition}

\begin{proof}
Assume that $(\mathcal{A}, \bullet, \omega)$ is a non-trivial pre-symplectic left-symmetric algebra with $\dim \mathcal{A} = 2$. If the commutator ideal $[\mathcal{A}, \mathcal{A}]$ is nondegenerate, Theorem~\ref{  ideal non degenere} implies that $(\mathcal{A}, \bullet, \omega)$ is a Milnor pre-symplectic algebra, which leads to a contradiction since non-trivial Milnor pre-symplectic algebras do not exist in dimension $2$. Thus, $[\mathcal{A}, \mathcal{A}]$ must be degenerate with respect to $\omega$, meaning that $
I := [\mathcal{A}, \mathcal{A}] \cap [\mathcal{A}, \mathcal{A}]^\perp \neq \{0\}.$ Consequently, $\dim I = 1$.   According to Theorem~\ref{double-ex1}, $(\mathcal{A}, \bullet, \omega)$ is the double extension of the trivial algebra $\{0\}$ by a one-dimensional algebra $\mathfrak{h}$ by means of $(0,0,0, \lambda)$. By formula~\eqref{Produit1}, the product and the bilinear form are given by
\[
d \bullet e = \lambda e, \quad d \bullet d = \lambda d, \quad \text{and} \quad \omega(e, d) = 1.
\]
If $\lambda = 0$, then the algebra would be trivial, contradicting our assumption. Hence, $\lambda \neq 0$. Now, let us perform a change of basis by setting:
\[
e_1 := \frac{1}{\lambda} d \quad \text{and} \quad e_2 := -\lambda e.
\]
With respect to this new basis $\{e_1, e_2\}$, the non-zero products and the form $\omega$ become:
\[
e_1 \bullet e_1 = e_1, \quad e_1 \bullet e_2 = e_2, \quad \text{and} \quad \omega(e_1, e_2) = 1.
\]
 This shows that $(\mathcal{A}, \bullet, \omega)$ is symplectic-isomorphic to $(\mathcal{A}_2, \omega)$.
\end{proof}

We now focus on the classification in dimension $4$. According to Theorem~\ref{Milnor0}, any pre-symplectic left-symmetric algebra is either a Milnor pre-symplectic algebra or can be obtained from a Milnor pre-symplectic algebra via a finite sequence of double extensions by a  commutative associative algebras.

Let $(\mathcal{A}, \bullet, \omega)$ be a $4$-dimensional pre-symplectic left-symmetric algebra that arises as a double extension of a lower-dimensional pre-symplectic left-symmetric algebra $(\mathcal{B}, \bullet_{\mathcal{B}}, \omega_{\mathcal{B}})$ by a one- or two-dimensional commutative associative algebra $(\mathfrak{h}, \circ)$.

If $\dim \mathfrak{h} = 1$, then $\dim \mathcal{B} = 2$. If $\dim \mathfrak{h} = 2$, then $\dim \mathcal{B} = 0$. In this case, $\mathcal{A} \simeq \mathfrak{h} \oplus \mathfrak{h}^*$, and the structural maps vanish identically, i.e.,
\[
\xi_i = D_i = a_0 = b_{ij} = l_i = 0, \quad \forall i,j \in \{1,2\}.
\]

We now consider a two-dimensional pre-symplectic left-symmetric algebra $(\mathcal{B},\bullet_{\mathcal{B}},\omega_{\mathcal{B}})$. We look for linear maps $\xi, D:\mathcal{B} \to \mathcal{B}$, where $\xi$ is symmetric with respect to $\omega_{\mathcal{B}}$, and an element $b_0 \in \mathcal{B}$ such that $(\xi, D, b_0, \lambda)$ satisfies~\eqref{Lesrelation dim1}.

When $\mathcal{B}$ is trivial, these equations reduce to:
\begin{equation}
D^2 = D \circ \xi, \qquad
\xi \circ D = \lambda D, \qquad
D^* \circ D = 0, \qquad
D^*(b_0)=0.
\label{abelian}
\end{equation}

\ssbegin{Proposition}\label{Solution F, G, b_0 symp}
Let $(\mathcal{B}, \omega_{\mathcal{B}})$ be a 2-dimensional trivial pre-symplectic algebra. Then $(\xi, D, b_0, \lambda)$ satisfies the compatibility equations \eqref{abelian} if and only if there exists a symplectic basis $\mathbb{B} = \{e_1, e_2\}$ of $\mathcal{B}$ such that $\xi$, $D$, and $b_0$ are expressed in $\mathbb{B}$ by one of the following mutually exclusive cases:
\begin{enumerate}
    \item[\rm (a)] $\xi = D = 0$, $b_0 = \alpha e_1 $, and $\lambda \in \mathbb{K}$, with $\alpha \in \mathbb{K}$.
    
    \item[\rm (b)] $\xi = 0$, $D = \begin{pmatrix} 0 & x \\ 0 & 0 \end{pmatrix}$, $b_0 = \alpha e_1$, and $\lambda = 0$, with $x \in \mathbb{K}^*$ and $\alpha \in \mathbb{K}$.
    
    \item[\rm (c)] $\xi = \begin{pmatrix} x & 0 \\ 0 & x \end{pmatrix}$, $D = \begin{pmatrix} x & y \\ 0 & 0 \end{pmatrix}$, $b_0 = \alpha e_1$, and $\lambda = x$, with $x \in \mathbb{K}^*$ and $y, \alpha \in \mathbb{K}$.
    
    \item[\rm (d)] $\xi = \begin{pmatrix} x & 0 \\ 0 & x \end{pmatrix}$, $D = 0$, $b_0 = \alpha e_1$, and $\lambda \in \mathbb{K}$, with $x \in \mathbb{K}^*$ and $\alpha \in \mathbb{K}$.
\end{enumerate}
\end{Proposition}

\begin{proof}
From the condition $\omega_{\mathcal{B}}(D(u), D(v)) = \omega_{\mathcal{B}}(u, D^*(D(v))) = 0$ for all $u,v \in \mathcal{B}$, we deduce that $\mathrm{Im}(D)$ is a totally isotropic subspace of $\mathcal{B}$. If $\mathrm{Im}(D) \neq \{0\}$, then since $\dim \mathcal{B} = 2$, we must have $\dim \mathrm{Im}(D) = 1$. Thus, we can choose a symplectic basis $\{e_1, e_2\}$ of $(\mathcal{B}, \omega_{\mathcal{B}})$ such that the matrix representing $D$ has the form:
\[
D = \begin{pmatrix} y & z \\ 0 & 0 \end{pmatrix}, \qquad \text{with } (y,z) \neq (0,0).
\]
On the other hand, since $\xi$ is  symmetric with respect to $\omega_{\mathcal{B}}$, its matrix representation in this symplectic basis must be scalar, taking the form:
\[
\xi = \begin{pmatrix} x & 0 \\ 0 & x \end{pmatrix}, \qquad \text{with } x \in \mathbb{K}.
\]

If $x = 0$, then $\xi = 0$. The equation $D^2 = D \circ \xi$ implies $D^2 = 0$, which yields $y = 0$. Since $(y,z) \neq (0,0)$, we have $z \neq 0$. Thus $\lambda = 0$. Since $D^*(b_0) = 0$, we obtain $b_0 = \alpha e_1$ for some $\alpha \in \mathbb{K}$, which corresponds exactly to case (b).

If $x \neq 0$, the identity $\xi \circ D = \lambda D$ yields $\lambda = x$. The equation $D^2 = D \circ \xi$ implies  $y^2 = xy$ and $yz = xz$. Since $(y,z) \neq (0,0)$, we obtain $y \neq 0$. Thus, we get $y = x$ and $z \in \mathbb{K}$. Since $D^*(b_0) = 0$, we get $b_0 = \alpha e_1$ for some $\alpha \in \mathbb{K}$, which produces case (c).

If $D = 0$, we choose an arbitrary symplectic basis $\{e_1, e_2\}$ of $(\mathcal{B}, \omega_{\mathcal{B}})$ such that $b_0=\al e_1$, with $\al\in \mathbb{K}$. Since $\xi$ is symmetric with respect to $\omega_{\mathcal{B}}$, it remains scalar:
\[
\xi = \begin{pmatrix} x & 0 \\ 0 & x \end{pmatrix}, \qquad \text{with } x \in \mathbb{K}^.
\]
If $x = 0$, we have $\xi= 0$, giving case (a). If $x\neq 0$, we directly obtain case (d).

\end{proof}

We treat now the non-trivial case. We consider the pre-symplectic left-symmetric algebra $(\mathcal{A}_2, \bullet_{\mathcal{A}_2}, \omega_{\mathcal{A}_2})$ defined in Proposition~\ref{mil3}, and we look for $(\xi, D, b_0, \lambda)$ satisfying system \eqref{Lesrelation dim1}.

\ssbegin{Proposition}\label{ntri}
Let $(\mathcal{A}_2, \bullet_{\mathcal{A}_2}, \omega_{\mathcal{A}_2})$ be the unique non-trivial pre-symplectic  left-symmetric algebra of dimension $2$. Then, $(\xi, D, b_0, \lambda)$ satisfies system \eqref{Lesrelation dim1} if and only if there exists a symplectic basis $\{e_1, e_2\}$ of $(\mathcal{A}_2, \omega_{\mathcal{A}_2})$ such that $\xi$, $D$, and $b_0$ are expressed by one of the following cases:
$$\xi = \begin{pmatrix} x & 0 \\ 0 & x \end{pmatrix}, \, D = \begin{pmatrix} x & 0 \\ y & 0 \end{pmatrix},\, b_0 = (x-\la)(xe_1+ye_2), \text{  and } \lambda \in \mathbb{K}, \text{ with } x, y \in \mathbb{K}.$$

    
    

\end{Proposition}

\begin{proof}
Recall the relations \eqref{Lesrelation dim1}:
\begin{equation*}
\begin{cases}
D([u,v]_{\bullet})=0,\; u\bullet D(v)=v\bullet D(u), \\
D(u)\bullet v)=\xi(u)\bullet v, 
\xi(u\bullet v)=u\bullet\xi(v),\\ 
\xi\circ D=\lambda D+\Rr^{\bullet}_{b_0}, \;
D^2=D\circ \xi, \;
D^*\circ D=0, \;
G^*(b_0)=0,
\end{cases}
\label{Lesrelation dim1-2}
\end{equation*}
for all \(u,v \in \A_2\), and $\xi$ is symmetric with respect to $\om_{\A_2}$.

According to Proposition~\ref{mil3}, there exists a symplectic basis $\{e_1, e_2\}$ of $(\mathcal{A}_2, \omega)$ such that the product on $\mathcal{A}_2$ is given by
\[
e_1\bullet e_1=e_1,\quad e_1 \bullet e_2 = e_2.
\]
Since $\xi$ is symmetric with respect to $\om_{\A_2}$, it follows that $\xi(e_1)=x e_1$ and $\xi(e_2)=x e_2$, with $x\in \mathbb{K}$.  On the other hand, from the relation $D([u,v]_{\bullet}) = 0$, we deduce that $D(e_2) = 0$. Using the relation  $D(u) \bullet v = \xi(u) \bullet v$, we obtain $D(e_1) = x e_1 + z e_2$, where $z \in \mathbb{K}$. Hence, the relations  $\xi(u \bullet v) = u \bullet \xi(v)$, $u \bullet D(v) = v \bullet D(u)$, $D^2 = D \circ \xi$, and $D^* \circ D = 0$ are automatically satisfied.

Now, let $b_0 = \alpha e_1 + \beta e_2$, where $\alpha, \beta \in \mathbb{K}$. From the relation $\xi \circ D = \lambda D + \Rr^{\bullet}_{b_0}$, we deduce that $\al=(x-\la)x$ and $\beta=(x-\la)z$.

    

    

\end{proof}

To establish the classification of 4-dimensional pre-symplectic left-symmetric algebras, we first recall the structural classification of 2-dimensional commutative associative algebras in the following proposition.
\ssbegin{Proposition}[\cite{Rakh}]\label{Rakhi}
Let $(\mathfrak{h}, \circ)$ be a  two-dimensional commutative associative algebra. Then, $(\mathfrak{h}, \circ)$ is isomorphic to one of the following non-isomorphic algebras:
\begin{enumerate}
\item[\rm (1)] $\mathfrak{h}_0$  is the trivial algebra.
    \item[\rm (2)] $\mathfrak{h}_1$ : 
    $ d_1 \circ d_1 = d_1. $
    \item[\rm (3)] $\mathfrak{h}_2$ : $ d_1 \circ d_1 = d_2. $
    
    \item[\rm (4)] $\mathfrak{h}_3$ : 
    $ d_1 \circ d_1 = d_1, \quad d_1\circ d_2=d_2\circ d_1=d_2,\quad d_2 \circ d_2 =t d_1,$ with $t\in \mathbb{K}.$
\end{enumerate}
\end{Proposition}

We have now all the needed ingredients to give the complete classification list of pre-symplectic left-symmetric algebras of dimension $4$.

\ssbegin{Theorem}\label{classif4}
Every $4$-dimensional pre-symplectic left-symmetric algebra $(\mathcal{A}, \bullet, \omega)$ is symplectic-isomorphic to one of the pre-symplectic left-symmetric algebras:
\begin{enumerate}
    \item[\rm (1)] $(\mathbb{K}^4, \omega)$ with the trivial product and $\omega = e_1^* \wedge e_2^* + e_3^* \wedge e_4^*$.
    
    \item[\rm (2)] $\mathcal{A}_{4,1} = \big(\mathrm{span}(x_1, x_2, y_1, y_2), \bullet, \omega\big)$ where:
    \[
    \begin{cases} 
    y_1 \bullet x_1 = x_1, \quad y_1 \bullet x_2 = x_2, \\
    \omega = y_1^* \wedge y_2^* + x_1^* \wedge x_2^*.
    \end{cases}
    \]

    \item[\rm (3)] $\mathcal{A}_{4,2} = \big(\mathrm{span}(d, e_1, e_2, e), \bullet, \omega\big)$ where:
    \[
    \begin{cases} 
     d \bullet e_2 = -\alpha e, \quad d \bullet e = \lambda e, \quad d \bullet d = \alpha e_1  + \lambda d, \\
    \omega = e^* \wedge d^* + e_1^* \wedge e_2^*, \quad \text{with } (\alpha, \lambda)\neq (0, 0).
    \end{cases}
    \]
    
    \item[\rm (4)] $\mathcal{A}_{4,3} = \big(\mathrm{span}(d, e_1, e_2, e), \bullet, \omega\big)$ where:
    \[
    \begin{cases} 
    e_2 \bullet e_2 = -x e, \quad e_2 \bullet d = x e_1, \quad d \bullet e_2 = -\alpha e, \quad d \bullet d = \alpha e_1, \\
    \omega = e^* \wedge d^* + e_1^* \wedge e_2^*, \quad \text{with } x \in \mathbb{K}^* \text{ and } \alpha \in \mathbb{K}.
    \end{cases}
    \]

    \item[\rm (5)] $\mathcal{A}_{4,4} = \big(\mathrm{span}(d, e_1, e_2, e), \bullet, \omega\big)$ where:
    \[
    \begin{cases} 
    e_1 \bullet e_2 = -x e, \quad e_2 \bullet e_2 = -y e, \quad e_1 \bullet d = x e_1, \quad e_2 \bullet d = y e_1, \\ 
    d \bullet e_1 = x e_1, \quad d \bullet e_2 = x e_2- \alpha e, \quad d \bullet e = x e, \quad d \bullet d = \alpha e_1 + x d, \\
    \omega = e^* \wedge d^* + e_1^* \wedge e_2^*, \quad \text{with } x \in \mathbb{K}^* \text{ and } y, \alpha \in \mathbb{K}.
    \end{cases}
    \]

    \item[\rm (6)] $\mathcal{A}_{4,5} = \big(\mathrm{span}(d, e_1, e_2, e), \bullet, \omega\big)$ where:
    \[
    \begin{cases} 
    d \bullet e_1 = x e_1, \quad d \bullet e_2 = x e_2- \alpha e, \quad d \bullet e = \la e, \quad d \bullet d = \alpha e_1 + \la d, \\
    \omega = e^* \wedge d^* + e_1^* \wedge e_2^*, \quad \text{with } x \in \mathbb{K}^* \text{ and } \la, \alpha \in \mathbb{K}.
    \end{cases}
    \]

    \item[\rm (7)] $\mathcal{A}_{4,6} = \big(\mathrm{span}(d, e_1, e_2, e), \bullet, \omega\big)$ where:
    \[
    \begin{cases} 
    e_1\bullet e_1=e_1+y e,\; e_1\bullet e_2=e_2-xe,\;d \bullet e_1 = x e_1 + (x-\la)y e, \; d \bullet e_2 = x e_2 - (x-\la) x e, \\ d \bullet e = \lambda e, \; e_1\bullet d=xe_1+y e_2,\; d \bullet d =(x-\la)( x e_1 + y e_2) + \lambda d, \\
    \omega = e^* \wedge d^* + e_1^* \wedge e_2^*, \quad \text{with } x,y, \lambda \in \mathbb{K}.
    \end{cases}
    \]

 \item[\rm (8)] $\mathcal{A}_{4,7} = \big(\mathrm{span}(d_1, d_2, e_1, e_2), \bullet, \omega\big)$ where:  $$\begin{cases} 
    d_1 \bullet d_1 =  l_1e_2,  \quad d_1 \bullet d_2 = -l_1 e_1, \quad
    d_2 \bullet d_1 = l_2 e_2, \quad d_2 \bullet d_2 = -l_2 e_1,\\
    \omega = e_1^* \wedge d_1^* + e_2^* \wedge d_2^*,  \text{ with } l_1, l_2\in \mathbb{K}.
\end{cases}$$
   
 \item[\rm (9)] $\mathcal{A}_{4,8} = \big(\mathrm{span}(d_1, d_2, e_1, e_2), \bullet, \omega\big)$ where: $$ \begin{cases} 
    d_1 \bullet e_1 = e_1, \quad 
    d_1 \bullet d_1 = l_1 e_2 + d_1, \quad d_1 \bullet d_2 = -l_1 e_1,  \\
    \omega = e_1^* \wedge d_1^* + e_2^* \wedge d_2^*,  \text{ with } l_1\in \mathbb{K}.
\end{cases}$$

 \item[\rm (10)] $\mathcal{A}_{4,9} = \big(\mathrm{span}(d_1, d_2, e_1, e_2), \bullet, \omega\big)$ where:  $$\begin{cases} 
     d_1 \bullet e_2 = e_1, \quad 
    d_1 \bullet d_1 = l_1 e_2 + d_2, \quad d_1 \bullet d_2 = -l_1 e_1, \\
    \omega = e_1^* \wedge d_1^* + e_2^* \wedge d_2^*,  \text{ with } l_1\in \mathbb{K}.
\end{cases}$$

 \item[\rm (11)] $\mathcal{A}_{4,10} = \big(\mathrm{span}(d_1, d_2, e_1, e_2), \bullet, \omega\big)$ where:  $$\begin{cases} 
     d_1 \bullet e_1 = e_1, \quad d_1 \bullet e_2 = e_2, \quad
    d_2 \bullet e_1 = t e_2, \quad d_2 \bullet e_2 = e_1, \quad
    d_1 \bullet d_1 = d_1, \\ \quad d_1 \bullet d_2 = d_2, 
    d_2 \bullet d_1 = l_2 e_2 + d_2, \quad d_2 \bullet d_2 = -l_2 e_1 + t d_1,\\
    \omega = e_1^* \wedge d_1^* + e_2^* \wedge d_2^*,  \text{ with } l_2, t\in \mathbb{K}.
\end{cases}$$

\end{enumerate}
\end{Theorem}

\begin{proof}
Let $(\mathcal{A}, \bullet, \omega)$ be a $4$-dimensional pre-symplectic left-symmetric algebra. Based on the degeneracy behavior of the commutator ideal $[\mathcal{A}, \mathcal{A}]$ with respect to $\omega$, we naturally distinguish two cases.

\textbf{Case 1:} $[\mathcal{A}, \mathcal{A}]$ is nondegenerate with respect to $\omega$.  
Then, according to Theorem~\ref{  ideal non degenere}, $(\mathcal{A}, \bullet, \omega)$ is a Milnor pre-symplectic left-symmetric algebra. In dimension $4$, by Example~\ref{ex4}, this implies that $(\mathcal{A}, \bullet, \omega)$ is symplectic-isomorphic to $(\mathcal{A}_{4,1}, \omega)$.

\textbf{Case 2:} $[\mathcal{A}, \mathcal{A}]$ is degenerate with respect to $\omega$.  
Then, by Theorem~\ref{Double-generale}, $(\mathcal{A}, \bullet, \omega)$ is obtained by a double extension of a lower-dimensional pre-symplectic left-symmetric algebra $(\mathcal{B}, \bullet_{\mathcal{B}}, \omega_{\mathcal{B}})$ by a commutative associative algebra $(\mathfrak{h}, \circ)$ of dimension $1$ or $2$. The extension is governed by the structural parameters $(D, \xi, b_0, \lambda)$ or $(D_i, \xi_i, a_0, b_{ij}, l_i, t_{ij}, s_{ij})$ for $i,j \in \{1,2\}$.

\begin{itemize}
    \item If $\dim \mathfrak{h} = 1$, then $\dim \mathcal{B} = 2$. According to the low-dimensional classification provided in Proposition~\ref{mil3}, the base algebra $(\mathcal{B}, \bullet_{\mathcal{B}})$ is either the trivial  algebra or is symplectic-isomorphic to the unique non-trivial  $(\mathcal{A}_2, \omega_{\mathcal{A}_2})$. 
    
    By applying the solution spaces determined in Propositions~\ref{Solution F, G, b_0 symp} and \ref{ntri} respectively, and substituting them into the general double extension product formula \eqref{Produit1}, we exhaustively obtain the algebras $\mathcal{A}_{4,2}, \dots, \mathcal{A}_{4,6}$.
    
  \item If $\dim \mathfrak{h} = 2$, then $\dim \mathcal{B} = 0$, and therefore $(\mathcal{A}, \bullet, \omega) \simeq (\mathfrak{h} \oplus \mathfrak{h}^*, \bullet, \omega)$. According to Proposition~\ref{Rakhi}, we rely on the classification of two-dimensional commutative associative algebras. Then, by solving the system~\eqref{Les relation dim2}, we deduce the scalars $t_{ij}$, $s_{ij}$, and $l_i$. Finally, by substituting these into the product formula~\eqref{Produit de dim2}, we obtain the algebras $\mathcal{A}_{4,7}, \dots, \mathcal{A}_{4,10}$.
\end{itemize}

\end{proof}

\ssbegin{Remark}
The list of pre-symplectic left-symmetric algebras in Theorem \ref{classif4} is exhaustive, but may contain redundancies. It would be interesting to determine for which values of the parameters these algebras are isomorphic to one another. Carrying out this classification does not seem straightforward and requires computational techniques that are beyond the scope of this paper. 
\end{Remark}

\bibliographystyle{elsarticle-num}

\end{document}